\documentclass[11pt]{article}
\usepackage{amsmath,amsthm,amssymb,comment}
\usepackage{booktabs}
\usepackage{graphicx, color}
\usepackage{psfrag}
\usepackage{epsf}
\usepackage{array}
\usepackage{caption}
\usepackage[mathscr]{eucal}
\usepackage[T1]{fontenc}
\usepackage[utf8]{inputenc}
\usepackage{subcaption}
\usepackage{authblk}
\usepackage{float}

\newcommand{\prob}{\mathbb P}
\newcommand{\E}{\mathbb E}
\newtheorem{theorem}{Theorem}[section]

\newtheorem{proposition}{Proposition}[section]
\newtheorem{remark}{Remark}[section]

\newcommand{\W}{W^{(q+\lambda)}}
	
\numberwithin{equation}{section}
\begin{document}
	
	

\title{A dual risk model with additive and proportional gains:
\\
ruin probability and dividends}

\author{Onno Boxma\thanks{Department of Mathematics and Computer Science, Eindhoven
	University of Technology, P.O. Box 513, 5600 MB Eindhoven, The
		Netherlands (o.j.boxma@tue.nl)\\
The research of Onno Boxma was supported via a TOP-C1 grant of the Netherlands Organisation for Scientific Research.}
	\hspace{2mm},
	 Esther Frostig\thanks{Department of Statistics, Haifa University, Haifa, Israel (frostig@stat.haifa.ac.il)
\\
The research of Esther Frostig was supported by the Israel Science Foundation, grant no. 1999/18.}
	\hspace{2mm}and
	 Zbigniew Palmowski\thanks{Department of Applied Mathematics, Wroclaw University of Science and Technology, Wroclaw, Poland (zbigniew.palmowski@gmail.com).\\
The research of Zbigniew Palmowski is partially supported by Polish National Science Centre Grant
No. 2016/23/B/HS4/00566 (2017-2020).}
}

		\maketitle
	\begin{abstract}
We consider a dual risk model with constant expense rate and i.i.d. exponentially distributed gains $C_i$ ($i=1,2,\dots$)
that arrive according to a renewal process with general interarrival times. We add to this classical dual risk model the proportional
gain feature, that is, if the surplus process just before the $i$th arrival is at level $u$, then
for $a>0$ the capital jumps up to the level $(1+a)u+C_i$.
The ruin probability and the distribution of the time to ruin are determined.
We furthermore identify the value of discounted cumulative dividend payments, for the case of a Poisson arrival process of proportional gains.
In the dividend calculations, we also consider a random perturbation of our basic risk process modeled by an
independent Brownian motion with drift.
\end{abstract}
\section{Introduction}

We consider a dual risk model with constant expense rate normalized at $1$.
Gains arrive according to a renewal process $\{N(t), t \geq 0\}$
with i.i.d. interarrival times $T_{i+1}-T_i$ having distribution $F(\cdot)$, density $f(\cdot)$ and  Laplace-Stieltjes transform (LST)
$\phi(\cdot)$.
If the surplus process just before the $i$th arrival is at level $u$, then the capital jumps up to the level $(1+a)u+C_i$,
$i=1,2,\dots$,
where $a>0$ and $C_1,C_2,\dots$  are i.i.d.\ exponentially distributed random variables with mean $1/\mu$.
Let $U(t)$ be the surplus process, with  $U(0)=x>0$, then we can write
\begin{equation}\label{riskprocess}
U(t)=x-t+ \sum_{i=1}^{N(t)} (C_i+aU(T_i-)), ~~~t \geq 0. \end{equation}
Taking $a=0$ yields a classical dual risk model, while $C_i \equiv 0$ yields a dual risk model with proportional gains.
$U(t)$ can also represent the workload in an M/G/1 queue
or the inventory level in a storage model or dam model
with a constant demand rate
and occasional inflow
that depends proportionally (apart from independent upward jumps)
on the current amount of work in the system.
We also give some results on a generalization of the model of \eqref{riskprocess} where
at the $i$th jump epoch the jump has size $au+C_i$ with probability $p$,
and has size $D_i$ with probability $1-p$, where $D_1,D_2,\dots$ are independent, exp($\delta$) distributed random variables, independent of $C_1,C_2,\dots$.

In this paper we are interested (i) in exactly identifying the Laplace transforms of the ruin probability and the ruin time, and (ii) in
approximating the value function, being the cumulative discounted amount of dividends paid up to the ruin time under a fixed barrier strategy.
To find this value function we solve a two-sided exit problem for the risk process \eqref{riskprocess}, which seems to be interesting in itself.
In the discounted dividend case we also add to the risk process \eqref{riskprocess} a perturbation modeled by a Brownian motion $X(t)$ with drift, that is,
we replace the negative drift $-t$ by process $X(t)$.

More formally, we start from the analysis of the ruin probability
\begin{equation}\label{ruinprob}
R(x):=\prob_x(\tau_x<\infty),
\end{equation}
where $\prob_x(\cdot):=\prob(\cdot|U(0)=x)$ and the ruin time is defined as the first time the surplus process equals zero:
\begin{equation}\label{taux}\tau_x = \inf\{t\geq 0: U(t)= 0\}.\end{equation}
Our method of analyzing $R(x)$ is based on a one-step analysis where the process under consideration is viewed at successive claim times.
We obtain the Laplace transform (with respect to initial capital) of the ruin probability for the risk process \eqref{riskprocess}.
We also analyze the double Laplace transform of the ruin time (with respect to initial capital and time).

Another quantity of interest for insurance companies is the expected cumulative and discounted amount of dividend payments
calculated under a barrier strategy.
To approach the dividend problem for the barrier strategy with barrier $b$, we consider the controlled surplus process $U^b$
satisfying
\begin{equation}\label{regproc}
U^b(t)=x-t+ \sum_{i=1}^{N(t)} (C_i+aU^b(T_i-))-L^b(t),
\end{equation}
where
the cumulative amount of dividends $L^b(t)$ paid up to time $t$ comes from paying all the overflow above a fixed level $b$
as dividends to shareholders.
The object of interest is the average value of the cumulative discounted dividends paid up to the ruin time:
\begin{equation}\label{vpi}
v(x):=\E_x\left[\int_0^{\tau_x^b}{\rm e}^{-qt}{\rm d}L^b(t) \right],
\end{equation}
where $\tau_x^b:=\inf\{t\geq 0: U^b(t)= 0\}$ is the ruin time and $q\geq 0$ is a given discount rate.
Here we adopt the convention that $\E_x$ is the expectation with respect to $\prob_x$.
We derive a differential-delay equation for $v(x)$. However, such differential-delay equations are notoriously difficult to solve, and
we were not able to solve our equation. Hence we have developed the following approach.
Under the additional assumption that $\{N(t), t \geq 0\}$ is a Poisson process and all
$C_i$ equal zero, we shall find
the expected cumulative discounted dividends
\begin{equation}\label{vN}
v_{N}(x):= \E_x\left[\int_0^{\tau_x^b(N)}{\rm e}^{-qt}{\rm d}L^b(t) \right] ,
\end{equation}
paid under the barrier strategy until reaching $\frac{b}{(a+1)^{N}}$, that is, up to
$\tau_x^b(N):=\inf\{t \geq 0: U^b(t)=\frac{b}{(a+1)^N}\}$.
By taking $N$ sufficiently large we can approximate the value function $v(x)$ closely by $v_N(x)$.
To find $v_N(x)$  we
first develop a method to solve a two-sided exit problem.
Defining
$d_{n}$  as the first time that $U^b$ reaches (down-crosses) $\frac{b}{(a+1)^{n}}$ and $u_{n}$ as the first time $U^b$ up-crosses $\frac{b}{(a+1)^{n}}$,
we determine (with $1_{\cdot}$ denoting an indicator function)
\begin{equation}\label{rhoN}
\rho_{N}(x):=\E_{x}\left[ e^{-qd_{N}}1_{d_{N}<u_{ 0}}\right] ,
\end{equation}
which seems of interest in its own.
We then use a very similar method to find $v_N(x)$, and to also solve a second two-sided exit problem, determining
\begin{equation}\label{muN}
\mu_{N}(x):=\E_{x}[e^{-qu_{0}}1_{u_{0}<d_{N}}].\end{equation}
In Section~\ref{sec:Brownian} we also perform a
similar analysis for the risk process \eqref{riskprocess} perturbed by
an independent Brownian motion.
There we also use the fluctuation theory
of spectrally negative L\'evy processes, expressing the exit identities in terms of
so-called scale functions, as presented for example in Kyprianou \cite{kyprianou2006}.

The dual risk model has been in the focus of actuarial science for some time.
In this model a company which continuously pays expenses, relevant to research or labour and operational costs,
occasionally gains some random income from selling a product or some inventions or discoveries \cite{avanzi2007, avanzi2009, bayraktar2012,Ng, yinwen2013,yinwen2014}.
As an example one can consider pharmaceutical or petroleum companies, R\&D companies, real estate agent offices or brokerage firms that sell mutual funds or insurance products with a front-end load.
For more detailed information, we refer the reader to \cite{avanzi2008}.
Lately, budgets of many start-ups or e-companies have shown a different feature.
Namely, their gains are not additive but strongly depend on the amount of investments, which usually are so huge that they are proportional
to the value of the company. Then the arrival gain is proportional not only to the investments but also to the value of the company.
Maybe the most transparent case is the example of CD projekt, one of the biggest Polish companies producing computer games.
Issuing new editions of its most famous game 'Witcher' produces jumps in the value of the company (which
is translated into jumps of asset value) and these jumps are proportional to the prior jump position of the value process;
see Figure 1.

\begin{figure}[htbp]
\begin{center}
\begin{minipage}{1.0\textwidth}
\centering
\begin{tabular}{cc}
\includegraphics[scale=0.5]{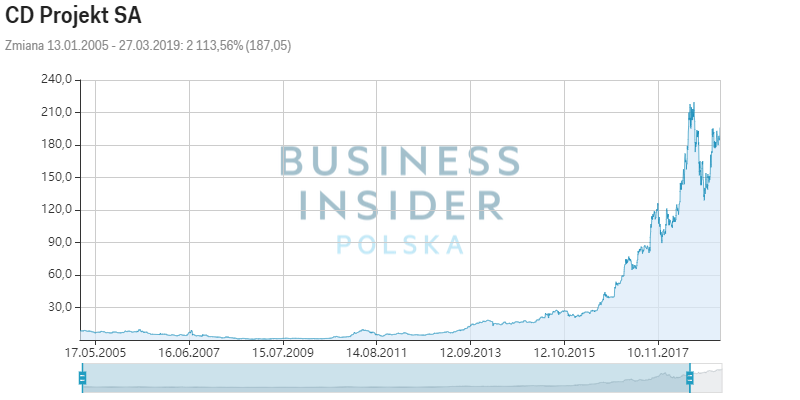}
 \end{tabular}
 \caption{Cd Projekt asset value.}
\end{minipage}
\label{cd}
\end{center}
\end{figure}

\noindent
{\bf Related literature}.
Not many papers consider
the ruin probability for the classical dual risk process (without proportional gain mechanism),
but it corresponds to the first busy period
in a single server queue with initial workload $x$ and as such we can refer to
\cite{Cohen, Prahbu}. If the interarrival time has an exponential distribution
then one can apply fluctuation theory of L\'evy processes to identify the Laplace transform
of the ruin time as well, see e.g. Kyprianou \cite{kyprianou2006}.
Albrecher et al.\ \cite{Albrecher} study the ruin probability in the dual risk model under a
loss-carry forward tax system and assuming exponentially distributed jump sizes.
Palmowski et al. \cite{Paletal} focus on a discrete-time set-up and study the finite-time ruin probability.
In terms of analysis technique, the approach in Sections~\ref{sec:ruin prob} and \ref{sec:timetoruin} bears similarities to the approach used in
\cite{BMR,BLM,BLMP,Vlasiou} to study
Lindley-type recursions $W_{n+1} = {\rm max}(0,a W_n+X_n)$,
where $a=1$ in the classical setting of a single server queue with $W_n$ the waiting time of the $n$th customer.

There is a good deal of work on dividend barriers in the dual model.
All of those papers assume that the cost function is constant, and gains are modeled by a compound Poisson process.
Avanzi et al. \cite{avanzi2007} consider cases where profits or gains follow an exponential distribution or a mixture of exponential distributions and they derive
explicit formulas for the expected discounted dividend value; see also Afonso et al. \cite{Afonso}. Avanzi and Gerber \cite{avanzi2008}
use the Laplace transform method
to study a dual model perturbed by a diffusion.
Bayraktar et al. \cite{bayraktar2012} and Avanzi et al. \cite{avanzi2016} employ fluctuation theory to prove the optimality of a barrier strategy for all spectrally positive L\'{e}vy processes
and express the value function in terms of scale functions.
Yin et al. \cite{yinwen2013, yinwen2014} consider terminal costs and
dividends that are paid continuously at a constant rate (that might be bounded from above) when the surplus is above that barrier; see also Ng \cite{Ng} for similar considerations.
Albrecher et al. \cite{Albrecher} examine a dual risk model in the presence of tax payments.
Marciniak and Palmowski \cite{MarPal} consider a more general dual risk process
where the rate of the costs depends on the present amount of reserves.
Boxma and Frostig \cite{BF} consider the time to ruin and the expected discounted dividends for a different  dividend policy, where a certain part of the gain is paid as dividends  if upon arrival the  gain finds the surplus above a barrier $b$ or if it  would bring the surplus above that level.
\\

\noindent
{\bf Organization of the paper}.
Section~\ref{sec:ruin prob}
is devoted to the determination of the ruin probability,
while Section~\ref{sec:timetoruin} considers the law of the ruin time.
Section~\ref{sec5} considers two-sided exit problems that allow one to find the ruin probability and
the total discounted dividend payments for the special case that the only capital growth is proportional growth.
In Section \ref{sec:Brownian} we handle the Brownian perturbation of the risk process
\eqref{riskprocess}.
Section~\ref{suggest} contains suggestions for further research.

\section{The ruin probability}\label{sec:ruin prob}

In this section we determine the Laplace transform of the ruin probability $R(x)$ when starting in $x$, as defined in \eqref{ruinprob}.
By distinguishing the two cases in which no jump up occurs before $x$ (hence ruin occurs at time $x$)
and in which a jump up occurs at some time $t \in (0,x)$, we can write:
\begin{equation}
R(x) = 1-F(x) + \int_{t=0}^x \int_{y=0}^{\infty} R((1+a)(x-t)+y) \mu {\rm e}^{-\mu y} {\rm d}y {\rm d}F(t) .
\label{eqstart}
\end{equation}
Introducing the Laplace transform \begin{equation}\label{LTruin}\rho(s) := \int_{x=0}^{\infty} {\rm e}^{-sx} R(x) {\rm d}x,
\end{equation} we have
\begin{equation}
\rho(s) = \frac{1-\phi(s)}{s}
+ \int_{x=0}^{\infty} {\rm e}^{-sx} \int_{t=0}^{x} \int_{z=(1+a)(x-t)}^{\infty} R(z) \mu {\rm e}^{-\mu z}
{\rm e}^{\mu(1+a)(x-t)} {\rm d}z  {\rm d}F(t) {\rm d}x .
\label{eq1}
\end{equation}
The triple integral in the righthand side of (\ref{eq1}), $I(s)$, can be rewritten as follows.
\begin{eqnarray}
I(s) &=& \int_{t=0}^{\infty} {\rm e}^{-st}
\int_{x=t}^{\infty} {\rm e}^{-s(x-t)} {\rm e}^{\mu(1+a)(x-t)}
\int_{z=(1+a)(x-t)}^{\infty} \mu {\rm e}^{-\mu z} R(z)
{\rm d}z
{\rm d}x
{\rm d}F(t)
\nonumber
\\
&=& \phi(s) \int_{v=0}^{\infty} {\rm e}^{-sv + \mu(1+a)v} \int_{z=(1+a)v}^{\infty}
\mu {\rm e}^{-\mu z} R(z) {\rm d}z {\rm d} v
\nonumber
\\
&=& \phi(s) \int_{z=0}^{\infty} \mu {\rm e}^{-\mu z} R(z)
\frac{{\rm e}^{(\mu(1+a)-s) \frac{z}{1+a}}-1}{\mu(1+a)-s} {\rm d} z .
\end{eqnarray}
Hence
\begin{equation}
\rho(s) = \frac{1-\phi(s)}{s} + \phi(s) \frac{\mu}{\mu(1+a)-s} \left[\rho(\frac{s}{1+a})-\rho(\mu)\right] .
\label{eq2}
\end{equation}
Introducing
\begin{equation}\label{HJ}
H(s) := \frac{1-\phi(s)}{s} - \phi(s) \frac{\mu}{\mu(1+a)-s} \rho(\mu), ~~~ J(s) := \phi(s) \frac{\mu}{\mu(1+a)-s},
\end{equation}
we rewrite (\ref{eq2}) into
\begin{equation}
\rho(s) = J(s) \rho\left(\frac{s}{1+a}\right) + H(s) .
\end{equation}
Thus $\rho(s)$ is expressed into $\rho\left(\frac{s}{1+a}\right)$,
and after $N-1$ iterations this results in
\begin{equation}
\rho(s) = \sum_{k=0}^{N-1} \prod_{j=0}^{k-1} J\left(\frac{s}{(1+a)^j}\right) H\left(\frac{s}{(1+a)^k}\right) +
\rho\left(\frac{s}{(1+a)^N}\right)
\prod_{j=0}^{N-1} J\left(\frac{s}{(1+a)^{j}}\right) ,
\label{iter}
\end{equation}
an empty product being equal to $1$.
Observe that, for large $k$, $H\left(\frac{s}{(1+a)^k}\right)$ approaches some constant
and $J\left(\frac{s}{(1+a)^k}\right)$ approaches $\frac{\phi(0) }{1+a} = \frac{1}{1+a} < 1$.
Hence the
$\sum_{k=0}^{N-1} \prod_{j=0}^{k-1}$ term in (\ref{iter}) converges geometrically fast, and we obtain
\begin{equation}
\rho(s) = \sum_{k=0}^{\infty} \prod_{j=0}^{k-1} J\left(\frac{s}{(1+a)^j}\right) H\left(\frac{s}{(1+a)^k}\right) .
\label{iter2}
\end{equation}
$\rho(\mu)$, featuring in the expression for $H(s)$, is still unknown. Taking $s=\mu$ in (\ref{iter2}) gives
\begin{equation}
\rho(\mu) = \sum_{k=0}^{\infty} \left(
\prod_{j=0}^{k-1} J\left(\frac{\mu}{(1+a)^j}\right) \right)
\left[ \frac{1 - \phi(\frac{\mu}{(1+a)^k})}{\frac{\mu}{(1+a)^k}} -
\phi\left(\frac{\mu}{(1+a)^k}\right) \frac{\mu}{\mu(1+a) - \frac{\mu}{(1+a)^k}} \rho(\mu) \right] ,
\label{rhomu}
\end{equation}
and hence
\begin{equation}\label{murho}
\rho(\mu) =
\frac{
\sum_{k=0}^{\infty} \big (
\prod_{j=0}^{k-1} J(\frac{\mu}{(1+a)^j}) \big )
\frac{1 - \phi(\frac{\mu}{(1+a)^k})}{\frac{\mu}{(1+a)^k}}
}
{1+
\sum_{k=0}^{\infty} \big (
\prod_{j=0}^{k-1} J(\frac{\mu}{(1+a)^j}) \big )
\phi(\frac{\mu}{(1+a)^k}) \frac{(1+a)^k}{(1+a)^{k+1} - 1}
}
.
\end{equation}
We can sum up
our analysis in the following first main result.
\begin{theorem}
The Laplace transform of the ruin probability,
$\rho(s) = \int_0^{\infty} {\rm e}^{-sx} \prob_x(\tau_x < \infty) {\rm d}x$,
is given in \eqref{iter2} with $H$ and $J$ given in \eqref{HJ},
where $\rho(\mu)$ is identified in \eqref{murho}.
\end{theorem}

{\bf Remark 1.}
It should be noticed that $R(x) \equiv 1$ satisfies Equation (\ref{eqstart})
but this trivial solution is not always the ruin probability.
In fact, defining $X_n := U(T_n-)$, the surplus just before the $n$th jump epoch,
the discrete Markov chain $\{X_n, n \geq 1\}$
satisfies the affine recursion
$$X_{n}=(a+1)X_{n-1} +(C_n-(T_n-T_{n-1})).$$
If $a>0$ then from \cite[Thm. 2.1.3, p. 13]{Dareketal}
we have that with a strictly positive probability $X_n$ tends to $+\infty$.
Thus $R(x)<1$ if $a>0$.
If $a=0$ then we are facing a $G/M/1$ queue, whose busy period ends with probability one iff $- \phi'(0) \geq \frac{1}{\mu}$.
\\

\noindent
{\bf Remark 2.}
Both $H(s)$ and $J(s)$ have a singularity at $s=\mu (1+a)$, which suggests that
the expression for $\rho(s)$ in (\ref{iter2}) has a singularity for every $s=\mu (1+a)^{j+1}$, $j=0,1,\dots$.
However, $s=\mu(1+a)$ is a removable singularity, as is already suggested by the form of (\ref{eq2}), where $s=\mu(1+a)$ also is a removable singularity.
To verify formally that $s=\mu(1+a)$ is not a singularity of (\ref{iter2}), we proceed as follows (the same procedure can be applied for
$s=\mu (1+a)^{j+1}$, $j=1,2,\dots$).
Isolate the coefficients of the factor $\frac{1}{\mu(1+a)-s}$ in (\ref{iter2}). Their sum $C(s)$ equals:
\begin{equation}
C(s) := -\phi(s) \mu \rho(\mu) + \phi(s) \mu \sum_{k=1}^{\infty} \prod_{j=1}^{k-1} J\left(\frac{s}{(1+a)^j}\right) H\left(\frac{s}{1+a)^k}\right) .
\end{equation}
Introducing $k_1 := k-1$ and $j_1 :=j-1$, and using (\ref{iter2}), it is readily seen that $C(\mu(1+a))=0$.
\\

\noindent
{\bf Remark 3.}
For the case of Poisson arrivals, taking $F(x) = 1 - {\rm e}^{-\lambda x}$,
one gets a specific form for $\rho(s)$, indicating that $R(x)$ is a weighted sum of exponential terms.
\\

\noindent
{\bf Remark 4.}
Finally a remark about possible generalizations.
We could allow a hyperexponential-$K$ distribution for $C$, leading to $K$ unknowns $\rho(\mu_1),\dots,\rho(\mu_K)$
which can be found by taking $s=\mu_1,\dots,s=\mu_K$.
\\
We could also consider the following generalization of the model \eqref{riskprocess} considered so far as well:
when the $i$th jump upwards occurs  while $U(T_i-)=u$, that jump has size $au+C_i$ with probability $p$,
and has size $D_i$ with probability $1-p$, where $D_1,D_2,\dots$ are independent, exp($\delta$) distributed random variables, independent of $C_1,C_2,\dots$.
By taking $p=1$ we get the old model, while $a=0$, $\mu=\delta$ gives a classical dual risk model.
It is readily verified that, for this generalized model, (\ref{eq2})  becomes:
\begin{eqnarray}
\rho(s) &=& \frac{1-\phi(s)}{s} + p \phi(s) \frac{\mu}{\mu(1+a)-s} \left[\rho\left(\frac{s}{1+a}\right)-\rho(\mu)\right]
\nonumber
\\
&+& (1-p) \phi(s) \frac{\delta}{\delta - s} [\rho(s) - \rho(\delta)] .
\label{eq3}
\end{eqnarray}
Introducing
\begin{equation}
H_1(s) := \frac{
\frac{1-\phi(s)}{s} -p \phi(s) \frac{\mu}{\mu(1+a) -s} \rho(\mu) - (1-p) \phi(s) \frac{\delta}{\delta-s} \rho(\delta)}
{1 - (1-p) \frac{\delta}{\delta -s} \phi(s)} ,
\end{equation}
\begin{equation}
J_1(s) := \frac{
p \phi(s) \frac{\mu}{\mu(1+a) -s}}
{1 - (1-p) \frac{\delta}{\delta -s} \phi(s)} ,
\end{equation}
we rewrite (\ref{eq3}) into
\begin{equation}
\rho(s) = J_1(s) \rho\left(\frac{s}{1+a}\right) + H_1(s),
\end{equation}
resulting in
\begin{equation}
\rho(s) = \sum_{k=0}^{\infty} \prod_{j=0}^{k-1} J_1\left(\frac{s}{(1+a)^j}\right) H_1\left(\frac{s}{(1+a)^k}\right) .
\label{iter2a}
\end{equation}
Finally, $\rho(\mu)$ and $\rho(\delta)$ have to be determined. One equation is supplied by substituting $s=\mu$ in (\ref{iter2a}) (just as was done below (\ref{iter2})).
For a second equation we invoke Rouch\'e's theorem, which implies that, for any $p \in (0,1)$, the equation $\delta - s - (1-p) \delta \phi(s) = 0$ has exactly one zero, say $s_1$, in the right-half $s$-plane.
Observing that $\rho(s)$ is analytic in that half-plane, so that $\rho(s_1)$ is finite, it follows from (\ref{eq3}) that
\[
\frac{1-\phi(s_1)}{s_1} + p \phi(s_1) \frac{\mu}{\mu (1+a) - s_1}[\rho(\frac{s_1}{1+a}) - \rho(\mu)] - (1-p) \frac{\delta}{\delta - s_1} \phi(s_1) \rho(\delta) = 0.
\]
While this provides a second equation, it also introduces a third unknown, viz., $\rho(\frac{s_1}{1+a})$.
However, substituting $s = \frac{s_1}{1+a}$ in (\ref{iter2a}) expresses $\rho(\frac{s_1}{1+a})$ into $\rho(\mu)$ and $\rho(\delta)$, thus providing a third equation.

\section{The time to ruin}\label{sec:timetoruin}

In this section we study the distribution of $\tau_x$,
the time to ruin when starting at level $x$, as defined in \eqref{taux}. Following a similar approach as in the previous section,
again distinguishing between the first upward jump occurring before or after $x$,
we can write:
\begin{equation}
\E[{\rm e}^{-\alpha \tau_x}] = {\rm e}^{-\alpha x} (1-F(x)) +
\int_{t=0}^x \int_{y=0}^{\infty} {\rm e}^{-\alpha y}  \E[{\rm e}^{-\alpha \tau_{(1+a)(x-t)+y}}] \mu {\rm e}^{-\mu y} {\rm d}y {\rm d}F(t) .
\label{eq3.1}
\end{equation}
{\bf Remark 5.}
Taking $\alpha = 0$ yields the ruin probability $R(x)$.
In that respect, it would not have been necessary to present a separate analysis of $R(x)$; however, to improve the readability of the paper, we have chosen to demonstrate the analysis technique
first for the easier case of $R(x)$.
\\

\noindent
Introducing the Laplace transform
\begin{equation}\label{doubleLT}
\tau(s,\alpha) := \int_{x=0}^{\infty} {\rm e}^{-sx} \E \left[{\rm e}^{-\alpha \tau_x}\right] {\rm d}x,
\end{equation}
and using very similar calculations as those leading to (\ref{eq2}), we obtain:
\begin{equation}
\tau(s,\alpha) = \frac{1-\phi(s+\alpha)}{s+\alpha} + \phi(s+\alpha) \frac{\mu}{\mu (1+a) -s} \left[\tau\left(\frac{s}{1+a},\alpha\right) - \tau(\mu,\alpha)\right] .
\label{eq3.2}
\end{equation}
Introducing
\begin{equation}
H_1(s,\alpha) := \frac{1-\phi(s+\alpha)}{s+\alpha} - \phi(s+\alpha) \frac{\mu}{\mu(1+a)-s} \tau(\mu,\alpha), ~~~ J_1(s,\alpha) := \phi(s+\alpha) \frac{\mu}{\mu(1+a)-s},
\label{3.4}
\end{equation}
we rewrite (\ref{eq3.2}) into
\begin{equation}
\tau(s,\alpha) = J_1(s,\alpha) \tau \left(\frac{s}{1+a},\alpha\right) + H_1(s,\alpha) ,
\end{equation}
which after $N-1$ iterations yields (an empty product being equal to $1$):
\begin{equation}
\tau(s,\alpha) = \sum_{k=0}^{N-1} \prod_{j=0}^{k-1} J_1 \left(\frac{s}{(1+a)^j},\alpha\right) H_1\left(\frac{s}{(1+a)^k},\alpha\right) +
\tau\left(\frac{s}{(1+a)^N},\alpha\right)
\prod_{j=0}^{N-1} J_1\left(\frac{s}{(1+a)^{j}},\alpha\right) .
\label{iternew}
\end{equation}
Observe that, for large $k$, $H_1\left(\frac{s}{(1+a)^k},\alpha\right)$ approaches some function of $\alpha$
and $J_1\left(\frac{s}{(1+a)^k},\alpha\right)$ approaches $\frac{\phi(\alpha) }{1+a} < 1$.
Hence the
$\sum_{k=0}^{N-1} \prod_{j=0}^{k-1}$ term in (\ref{iternew}) converges geometrically fast, and we obtain
\begin{equation}
\tau(s,\alpha) = \sum_{k=0}^{\infty} \prod_{j=0}^{k-1} J_1\left(\frac{s}{(1+a)^j},\alpha\right) H_1\left(\frac{s}{(1+a)^k},\alpha\right) .
\label{iternew2}
\end{equation}
$\tau(\mu,\alpha)$, featuring in the expression for $H_1(s,\alpha)$, is still unknown. Taking $s=\mu$ in (\ref{iternew2}) gives
\begin{eqnarray}
&&\tau(\mu,\alpha) = \sum_{k=0}^{\infty} \left(
\prod_{j=0}^{k-1} J_1\left(\frac{\mu}{(1+a)^j},\alpha\right) \right)\nonumber\\&&\qquad
\left[ \frac{1 - \phi(\frac{\mu}{(1+a)^k}+\alpha)}{\frac{\mu}{(1+a)^k}+\alpha} -
\phi \left(\frac{\mu}{(1+a)^k}+\alpha\right) \frac{\mu}{\mu(1+a) - \frac{\mu}{(1+a)^k}} \tau(\mu,\alpha) \right] ,
\end{eqnarray}
and hence
\begin{equation}\label{taualpha}
\tau(\mu,\alpha) =
\frac{
\sum_{k=0}^{\infty} \big (
\prod_{j=0}^{k-1} J_1(\frac{\mu}{(1+a)^j},\alpha) \big )
\frac{1 - \phi(\frac{\mu}{(1+a)^k}+\alpha)}{\frac{\mu}{(1+a)^k}+\alpha}
}
{1+
\sum_{k=0}^{\infty} \big (
\prod_{j=0}^{k-1} J_1(\frac{\mu}{(1+a)^j}) \big )
\phi(\frac{\mu}{(1+a)^k}+\alpha) \frac{(1+a)^k}{(1+a)^{k+1} - 1}
}
.
\end{equation}
Thus we have proved the second main result of this paper:
\begin{theorem}
The double Laplace transform (with respect to time and initial conditions)
$\tau(s,\alpha) = \int_{x=0}^{\infty} {\rm e}^{-sx} \E \left[{\rm e}^{-\alpha \tau_x}\right] {\rm d}x$
is given in \eqref{iternew2} with $H_1$ and $J_1$ given in
\eqref{3.4}, with $\tau(\mu,\alpha)$ identified in \eqref{taualpha}.
\end{theorem}

\section{Exit problems, ruin and barrier dividend value function}
\label{sec5}

In this section we consider the same model as in the previous sections, but with the restriction
that the only growth is a proportional growth occurring according to a Poisson process at rate $\lambda$;
throughout this section we further assume that $C_{i}\equiv 0$.
For this model we solve the two-sided downward exit problem in Subsection~\ref{sec5.1}.
In Subsection~\ref{sec5.2} we subsequently use a similar method to determine
the discounted cumulative dividend payments paid up to the ruin time under the barrier
strategy with barrier $b$ and with a discount rate $q$.
But first we briefly discuss an alternative, more straightforward, approach to determining the discounted cumulative dividend payments,
pointing out why this approach does not work.
We start from the observation that for $x>b$ we have
\begin{equation}\label{upperregion}
v(x)=v(b)+x-b.\end{equation}
We now focus on $x\leq b$.
One-step analysis based on the first arrival epoch gives:
\begin{equation}
v(x)=\int_0^x\lambda {\rm e}^{-(\lambda+q)t} v((x-t)(1+a)){\rm d} t, ~~~0 \leq x<b,
\label{eqeq}
\end{equation}
and by taking $z:=x-t$ and taking the derivative with respect to $x$
we end up with the equation
\begin{equation}
v'(x)+(\lambda+q)v(x)=\lambda v(x(1+a)),\qquad x \leq \frac{b}{1+a}.
\label{num}
\end{equation}
Moreover, from \eqref{upperregion} we have
\begin{equation}
v'(x)+(\lambda+q)v(x)=\lambda (x(1+a)-b+v(b)),\qquad  \frac{b}{1+a} < x  \leq b,
\label{num2}
\end{equation}
which can be easily solved.
Unfortunately, differential-delay equations like \eqref{num} seems hard to solve explicitly; cf.\ \cite{Hale}.
As an alternative, one could try to solve the equation numerically.

Instead, we adopt an approach, distinguishing levels $L_n := \frac{b}{(a+1)^n}$ and assuming that ruin occurs when level $\frac{b}{(a+1)^N}$ is reached for some value of $N$.
When $N$ is large, the expected amount of discounted cumulative dividends closely approximates the expected amount until ruin at zero occurs.
The above choice of levels is very suitable because
each proportional jump upward brings the process from a value in $(L_{n+1},L_{n})$ to a value in $(L_{n},L_{n-1})$.


\subsection{Two-sided downward exit problem and ruin time}
\label{sec5.1}	

Recall that $d_{n}$  is the first time that $U$ reaches (down-crosses) $\frac{b}{(a+1)^{n}}$ and $u_{n}$ is the first time the
risk process up-crosses $\frac{b}{(a+1)^{n}}$. Hence $u_0$ is the time until
level $b$ is first up-crossed, i.e., dividend is being paid.
For an integer $N$ we aim to obtain the Laplace transform
$\rho_{N}(x)=\E_{x}\left[ e^{-qd_{N}}1_{d_{N}<u_{ 0}}\right]$
defined in \eqref{rhoN}.
For large $N$, $\rho_N(x)$ approximates the LST of the time until ruin in the event that no dividend is ever paid.
If $q=0$ then $\rho_N(x)$ approximates the probability that ruin occurs before dividends are paid, that is, before reaching $b$.
If $q=0$ and $N$ and $b$ are both tending to infinity then  $\rho_N(x)$ approximates the ruin probability $R(x)$ defined in
\eqref{ruinprob}.
For $1\leq n \leq N$ let
\begin{equation}
\rho_{n,N}:=\rho_{N}\left(\frac{b}{(a+1)^{n}}\right).
\end{equation}
To simplify notation we denote $\rho_{n}:=\rho
_{n,N}$.
Clearly, $\rho_{N}=1$.
As announced above, we consider levels $$L_n := \frac{b}{(a+1)^n},\quad n=0,1,\dots,N;$$
let
\[\mathcal{N}:=\left\{L_1,\dots,L_N \right\}. \]
To determine $\rho_N(x)$ when $x \in (\frac{b}{(a+1)^n},\frac{b}{(a+1)^{n-1}}] = (L_n,L_{n-1}]$, note that, because $d_n < u_0$, there must be a down-crossing of level $L_{n}$ before level $b$ is up-crossed.
We can now distinguish $n$ different possibilities:
when starting at $x$, the surplus process  first decreases through $L_{n}$, or it first increases via jumps above level $L_{n-j}$ before there is a first down-crossing through that same level $L_{n-j}$, $j=1,\dots,n-1$.
Denoting the time for the former event as
\begin{equation}
T_{n,0} := d_n 1_{d_n < u_{n-1}},
\end{equation}
and the times for the latter $n-1$ events as
\begin{equation}
T_{n,j}:=u_{n-1}1_{u_{n-1}<d_{n}}+u_{n-2}1_{u_{n-2}<d_{n-1}}+...+u_{n- j}1_{u_{n-j}<d_{n-j+1}}+d_{n-j}1_{d_{n-j}<u_{n-j-1}} , ~~~j=1,\dots,n-1,
\end{equation}
we can derive the following representation of $\rho_{N}$.
It will turn out to be useful to introduce
$$\bar{G}_{c}(t)=e^{-(\lambda+q)ct}\quad \text{and} \quad g_{c}(t)=\lambda c e^{-(\lambda+q)ct}.$$
\begin{theorem}
\label{thm5.1}
	For
	 $\frac{b}{(a+1)^{n}}<x\leq \frac{b}{(a+1)^{n-1}}$,
	 \begin{eqnarray}
	 \rho_{N}(x) &=&
	\bar{G}_{1}\left(x-\frac{b}{(a+1)^{n}}\right)\rho_{n}\nonumber\\
	&+&\bar{G}_{(a+1)}g \circledast _1\left(x-\frac{b}{(a+1)^{n}}\right)\rho_{n-1}\nonumber\\
	&+&\bar{G}_{(a+1)^{2}} \circledast g_{(a+1)} \circledast g_1\left(x-\frac{b}{(a+1)^{n}}\right)\rho_{n-2}\nonumber\\
	&+&...
\nonumber
\\
&+&\bar{G}_{(a+1)^{n-1}} \circledast g_{(a+1)^{n-2}} \circledast ... \circledast g_1\left(x-\frac{b}{(a+1)^{n}}\right)\rho_{1},
\label{eq.main}
		 \end{eqnarray}
		 where $\circledast$ denotes convolution.
\end{theorem}
\begin{proof}
	Let  $\frac{b}{(a+1)^{n}}<x\leq \frac{b}{(a+1)^{n-1}}$.
The $n$ terms in the righthand side of (\ref{eq.main}) represent the $n$ disjoint possibilities where $d_n < u_0$.
Notice that
$T_{n,j}$ is the first time that the process $U$ reaches a level in  $\mathcal{N}$ via a down-crossing, by reaching $\frac{b}{(a+1)^{n-j}}$.
Furthermore, $\rho_{n-j}$ is the Laplace transform of the time to reach  $\frac{b}{(a+1)^{N}}$ starting at $\frac{b}{(a+1)^{n-j}}$.
Now first considering $T_{n,0}$, we have	
	\begin{equation}\label{eq.dn}
	E_{x}\left[ e^{-qd_{n}}1_{d_{n}<u_{n-1}}\right]
= {\rm e}^{-(q+\lambda)(x - \frac{b}{(a+1)^n})}
	=\bar{G}_1\left(x-\frac{b}{(a+1)^{n}}\right) .
	\end{equation}
By the strong Markov property, considering $T_{n,j}$, we derive
		\begin{eqnarray}\label{eq.Tnj}
		&&\E_{x}\left[e^{-q(u_{n-1}+\dots+u_{n-j} +d_{n-j})} 1_{u_{n-1}<d_n,\dots,u_{n-j}<d_{n-j+1},d_{n-j}<u_{n-j-1}} \right]=\nonumber\\
		&&\nonumber\\
		&&\qquad\int_{t_{1}=0}^{A_{n,0}}\int_{t_{2}=0}^{A_{n,1}}...\int_{t_{j}=0}^{A_{n,j-1}}\bar{G}_1(A_{n,j})g_1(t_{j})g_1(t_{j- 1})...g_1(t_{1}) {\rm d}t_{j}... {\rm d}t_{1} ,
		\end{eqnarray}
where
               \begin{equation}
		A_{n,0}:=x-\frac{b}{(a+1)^{n}},
		\end{equation}
		and for $k=1,2,...,n$,
			\begin{equation}\label{eq.Ank}
		A_{n,k}:=(a+1)(A_{n,k-1}-t_{k}) ,
				\end{equation}
with $t_k$ an integration variable.
		By the change of variables $y_{j}=t_{j}/(a+1)^{j-1}$ we obtain that:
		\begin{eqnarray}\label{eq.Tnj1}
		&&\E_{x}\left[e^{-q(u_{n-1}+\dots+u_{n-j} +d_{n-j})} 1_{u_{n-1}<d_n,\dots,u_{n-j}<d_{n-j+1},d_{n-j}<u_{n-j-1}} \right]=\nonumber\\
		&&\qquad=\bar{G}_{(a+1)^{j}} \circledast g_{(a+1)^{j-1}} \circledast ... \circledast
		g_1\left(x-\frac{b}{(a+1)^{n}}\right) ,
\nonumber\\
		\end{eqnarray}
	and the theorem follows.	
		\end{proof}
		From Theorem \ref{thm5.1} it follows that to obtain $\rho_{n}$ we need to solve the following $N-1$ equations for $n=1,2,...,N-1$:
		\begin{eqnarray}
		&& \rho_{n}=
		\bar{G}_{1}\left(\frac{b}{(a+1)^{n}}-\frac{b}{(a+1)^{n+1}}\right)\rho_{n+1}\nonumber\\
		&&\qquad+\bar{G}_{(a+1)} \circledast g_1\left(\frac{b}{(a+1)^{n}}-\frac{b}{(a+1)^{n+1}}\right)\rho_{n}\nonumber\\
		&&\qquad+\bar{G}_{(a+1)^{2}} \circledast g_{(a+1)} \circledast g_1\left(\frac{b}{(a+1)^{n
		}}-\frac{b}{(a+1)^{n+1}}\right)\rho_{n-1}\nonumber\\
		&&\qquad+\quad ...
\nonumber
\\
&&\qquad+\bar{G}_{(a+1)^{n}} \circledast g_{(a+1)^{n-1}} \circledast ... \circledast g_1\left(\frac{b}{(a+1)^{n
		}}-\frac{b}{(a+1)^{n+1
}}\right)\rho_{1}. \label{eq.main1}
		\end{eqnarray}
		Defining
		\[\gamma_{0}(x):=\bar{G}_1(x), \]
		and for $n\geq 1$,
		\[\gamma_{n}(x):=\bar{G}_{(a+1)^{n}} \circledast g_{(a+1)^{n-1}} \circledast ... \circledast g_1(x), \]
		then by the formula for convolution of exponentials given in \cite[Chap. 5]{Ross}
for $n\geq 1$:	
		\begin{equation}
		\gamma_{n}(x)=
	 \left(\frac{\lambda}{\lambda+q}\right)^n \sum_{i=0}^{n}\frac{e^{- (\lambda + q)(a+1)^{i}x}}{\prod_{j\neq i} ((a+1)^{j}-(a+1)^{i})}.
		\end{equation}
		Thus we have the following set of linear
		 equations for $n=1,2,...,N-1$:
\begin{equation}
		 \rho_{n}=\sum_{j=0}^{n}\gamma_{j}\left(\frac{b}{\left(a+1\right)^{n}}\left(1-\frac{1}{a+1}\right)\right)\rho_{n+1-j}.
\label{rhon1}
\end{equation}
	 Notice that $\rho_{N}=1$.
The above formula can be rewritten as follows:
\begin{equation}
\rho_n = \sum_{j=0}^n \gamma_{j,n} \rho_{n+1-j} = \sum_{j=1}^{n+1} \gamma_{n+1-j,n} \rho_j ,
\label{rhon}
\end{equation}
where
\begin{equation}
\gamma_{j,n} := \gamma_j\left(\frac{b}{(a+1)^n}\left(1 - \frac{1}{a+1}\right)\right).
\label{gammajn}
\end{equation}
Introducing the $(N-1) \times (N-1)$ matrix $\Gamma$, with as its $n$th row $(\gamma_{n,n}, \gamma_{n-1,n},\dots, \gamma_{0,n},0\dots,0)$,
and the column vector $\rho := (\rho_1,\dots,\rho_{N-1})^T$,
we can write the set of equations (\ref{rhon}) as
\begin{equation}
\rho = \Gamma \rho +  Z,
\end{equation}
where $Z =(0,\dots,0, \gamma_{0,N-1})^T$.
Hence, with $I$ the $(N-1) \times (N-1)$ matrix with ones on the diagonal and zeroes at all other positions,
\begin{equation}
\rho = (I - \Gamma)^{-1} Z .
\end{equation}

\subsection{Expected discounted dividends}
\label{sec5.2}

Recall that
$v_{N}(x)$ defined in \eqref{vN}
is the expected discounted dividends
under the barrier strategy  until reaching $\frac{b}{(a+1)^{N}}$, that is up to
$\tau_x^b(N)$ for the regulated process $U^b(t)$ defined in \eqref{regproc}.
Note that
\[v(x)=\lim_{N\rightarrow+\infty} v_N(x)\]
for $v(x)$ defined in \eqref{vpi}.
Let
\begin{equation}
	 v_{n}:=v_{N}\left(\frac{b}{(a+1)^{n}}\right)\quad\text{for $n=0,1,\dots,N-1$.}
\label{vndef}
\end{equation}
The next theorem identifies  $v_{N}(x)$.
	  	 \begin{theorem}
\label{thm5.2}
	 	For
	 	$\frac{b}{(a+1)^{n}}<x\leq \frac{b}{(a+1)^{n-1}}$,
	 	\begin{eqnarray}
	 	&& v_{N}\left(x\right)=
	 	\bar{G}_{1}\left(x-\frac{b}{\left(a+1\right)^{n}}\right)v_{n}\nonumber\\
	 	&&\qquad+\bar{G}_{\left(a+1\right)} \circledast g_1\left(x-\frac{b}{\left(a+1\right)^{n}}\right)v_{n-1}\nonumber\\
	 	&&\qquad+\bar{G}_{\left(a+1\right)^{2}} \circledast g_{\left(a+1\right)} \circledast g_1\left(x-\frac{b}{\left(a+1\right)^{n}}\right)v_{n-2}\nonumber\\
	 	&&\qquad+\quad...
\nonumber
\\
&&\qquad+\bar{G}_{\left(a+1\right)^{n-1}} \circledast g_{\left(a+1\right)^{n-2}} \circledast ... \circledast g_1\left(x-\frac{b}{\left(a+1\right)^{n}}\right)v_{1}\nonumber\\
		 	&&\qquad+1 \circledast	g_{\left(a+1\right)^{n-1}} \circledast ... \circledast g_1\left(x-\frac{b}{\left(a+1\right)^{n}}\right)v\left(b\right)\nonumber\\
	 		&&\qquad+\left(a+1\right)^{n}\mathcal{Q} \circledast g_{\left(a+1\right)^{n-1}} \circledast ... \circledast g_1\left(x-\frac{b}{\left(a+1\right)^{n}}\right),
	 	\label{eq.mainv}
	 	\end{eqnarray}
	 	where $\mathcal{Q}(x)=x$ and  $\circledast$ again denotes convolution.
	 \end{theorem}
\begin{proof}
The proof follows exactly the same reasoning as the proof of Theorem~\ref{thm5.1}:
we  consider again the disjoint events in which the first down-crossing of a level from $\mathcal{N}$ occurs at $L_n-j$,  $j=0,\dots,n-1$.
However, we now do not exclude the possibility that $L_0 = b$ is up-crossed  before level $L_N$ is reached.
This gives rise to the last two lines of (\ref{eq.mainv}). More precisely,
let $L_{n}<x\leq L_{n-1}$ and
$\mathcal{A}_{n}$ be  the event that level $ L_{0}=b$  is up-crossed before down-crossing one of the levels $L_{j}, j=1,...,N$. This event occurs when each of the following $n$ jumps occurs before down-crossing $L_{n-j+1}, j=0,...,n$, i.e. when $u_{n-j}<d_{n+1-j} ,j=1,...,n$. For $L_n<x\leq L_{n-1}$, the time to this event is
\begin{equation}
\Upsilon_{n}(x):=u_{n-1}1_{u_{n-1}<d_{n}}+u_{n-2}1_{u_{n-2}<d_{n-1}}+...+u_{1 }1_{u_{1}<d_{2}}+u_{0}1_{u_{0}<d_{1}}.
\end{equation}
The Laplace transform of $\Upsilon_{n}$ (or the discounted time until  $\mathcal{A}_{n}$  occurs) starting at $x$ with $L_{n}<x\leq L_{n-1}$ can be obtained by similar arguments as those leading to  (\ref{eq.Tnj1}), that is,
\begin{eqnarray}
&&\E_{x}\left[e^{-q(u_{n-1}+\dots+u_{0} )} 1_{u_{n-1}<d_n,\dots , u_{0}<d_{1}} \right]\nonumber\\
&&=	1 \circledast
g_{(a+1)^{n-1}} \circledast ... \circledast g_1\left(x-\frac{b}{(a+1)^{n}}\right).
\end{eqnarray}
Once the process up-crosses $b$, dividend is paid and the process restarts at level $b$. Thus the one-but-last
line of (\ref{eq.mainv}) is the expected discounted dividends paid until ruin  starting at time $u_{0}$ at level $b$ (not including the dividends paid at this time).
The expected discounted dividends paid at
$u_{n-1}+\dots + u_{0}$ is:
\begin{eqnarray}
&&\int_{t_{1}=0}^{A_{n,0}}\int_{t_{2}=0}^{A_{n,1}}...\int_{t_{n}=0}^{A_{n,n-1}}A_{n,n}g_1(t_{n})g_1(t_{n- 1})...g_1(t_{1}) {\rm d}t_{n}... {\rm d}t_{1}\nonumber\\
&&=(a+1)^{n}\mathcal{Q} \circledast g_{(a+1)^{n-1}} \circledast ... \circledast  g_{1}\left(x-\frac{b}{(a+1)^{n}}\right),
\end{eqnarray}
where $A_{n,n}$ is defined in (\ref{eq.Ank}) and the last equality is obtained by change of variables $y_{j}=\frac{t_{j}}{(1+a)^{j-1}}, j=1,2,...,n$.
\end{proof}

It remains to determine $v_0=v(b),v_1,\dots,v_{N-1}$, since then from Theorem~\ref{thm5.2} we have $v_N(x)$ for all $x \in (0,b]$.
Notice that $v_N=0$.
We first derive an equation for $v(b)$.
Let $$\delta_{n}:=(a+1)^{n}\mathcal{Q} \circledast g_{(a+1)^{n-1}} \circledast ... \circledast g_1\left(\frac{b}{(a+1)^{n}}\left(1-\frac{1}{a+1}\right) \right),$$
and let $$\omega_{n}:=1 \circledast g_{(a+1)^{n-1}} \circledast ... \circledast g_{1}\left(\frac{b}{(a+1)^{n-1}}\left(1-\frac{1}{a+1}\right)\right).$$
We distinguish between the following cases, when starting from $b$: (i) level $L_1 = \frac{b}{a+1}$ is reached before $b$ is up-crossed again;
this gives rise to the first term in the righthand side of (\ref{eqvb}) below;
(ii) $b$ is up-crossed before level $L_1$ is reached. Thus
\begin{equation}
v(b)=\bar{G}_1\left(b-\frac{b}{a+1}\right) v\left(\frac{b}{a+1}\right)
+ \frac{\lambda}{\lambda+q}\left(1-\bar{G}_1\left(b-\frac{b}{a+1}\right)\right)v(b)
+(a+1)\mathcal{Q} \circledast g_{1}\left(b-\frac{b}{a+1}\right).
\label{eqvb}
\end{equation}
Notice that $\frac{\lambda}{\lambda+q}(1-\bar{G}_1(b-\frac{b}{a+1}))=1 \circledast g_{1}\left(b-\frac{b}{a+1}\right)=\omega_{1}$.
Hence $v_{0}=\gamma_{0}(b-\frac{b}{a+1})v_{1} + \omega_{1}v_{0} + \delta_1$.
By taking $x = \frac{b}{(a+1)^{n-1}}$ in Theorem~\ref{thm5.2},
we get for $n= 1,\dots,N-1$,
\begin{equation}
v_{n}=\sum_{j=0}^{n}\gamma_{j}\left(\frac{b}{(a+1)^{n-1}}\left(1-\frac{1}{a+1}\right)\right) v_{n+1-j}+\omega_{n+1}v_{0}+
\delta_{n+1}.
\label{vn}
\end{equation}
Introducing
the column vector $V := (v_0,\dots,v_{N-1})^T$,
we can write the equation for $v_0$ and the set of equations (\ref{vn}) together as
\begin{equation}
V = \Psi V +  \Delta,
\end{equation}
where  $\Psi$ is an $N\times N$ matrix with row $n, n=0,\cdots,N-1$ equal to $(\omega_{n+1},\gamma_{n,n},\cdots,\gamma_{0,n},0,\cdots,0)$. Notice that row $N-1$ is $(\omega_{N},\gamma_{N-1,N-1,},\cdots,\gamma_{1,N-1})$.
Hence
\begin{equation}
V = (I - \Psi)^{-1} \Delta .
\end{equation}

\begin{remark}
\rm Our analysis can be used to solve the two-sided upward exit problem for our risk process as well.
We recall that $\mu_{n}=\mu_{N}\left(\frac{b}{(a+1)^n}\right)$ and it is defined in \eqref{muN}.
Then by the same arguments as those leading to (\ref{eq.mainv}) and (\ref{vn}), for
	$\frac{b}{(a+1)^{n}}<x\leq \frac{b}{(a+1)^{n-1}}$, we obtain that:	
	\begin{eqnarray}
	&& \mu_{N}(x)=
	\bar{G}_{1}\left(x-\frac{b}{(a+1)^{n}}\right)\mu_{n}\nonumber\\
	&&\qquad+\bar{G}_{(a+1)} \circledast g_1\left(x-\frac{b}{(a+1)^{n}}\right)\mu_{n-1}\nonumber\\
	&&\qquad+\bar{G}_{(a+1)^{2}} \circledast g_{(a+1)} \circledast g_1\left(x-\frac{b}{(a+1)^{n}}\right)\mu_{n-2}\nonumber\\
	&&\qquad+\quad ...
	\nonumber
	\\
	&&\qquad+\bar{G}_{(a+1)^{n-1}} \circledast g_{(a+1)^{n-2}} \circledast ... \circledast g_1\left(x-\frac{b}{(a+1)^{n}}\right)\mu_{1}\nonumber\\
	&&\qquad+1 \circledast g_{(a+1)^{n-1}} \circledast ... \circledast g_1\left(x-\frac{b}{(a+1)^{n}}\right) .
	\label{eq.mainmu}
	\end{eqnarray}
	Similar to the equation for $v_0$ and Equation (\ref{vn}) we obtain that
	\begin{eqnarray}
	&&\mu_{0}=\gamma_{0}(b-\frac{b}{a+1})\mu_{1}+\omega_{1},\\
	&&\mu_{n}=\sum_{j=0}^{n}\gamma_{j}\left(\frac{b}{(a+1)^{n-1}}\left(1-\frac{1}{a+1}\right)\right) \mu_{n+1-j}+\omega_{n+1}, ~~~ n=1,\dots,N-1.
	\end{eqnarray}
Moreover, observe that $\mu_0 = \mu_N(b)=1$.

\end{remark}

\section{Exit times and barrier dividends value function with Brownian perturbation}\label{sec:Brownian}
In this section we extend the model of Section 4 by allowing small perturbations between jumps.
These perturbations are modeled by
a Brownian motion $X(t)$  with drift $\eta$ and variance $\sigma^{2}$, that is,
\begin{equation}\label{linearBrownian}
X(t)=\eta t +\sigma B(t), \end{equation}
for a standard Brownian motion $B(t)$.
Hence our risk process is formally defined as
\begin{equation}\label{riskprocessB}
U(t)=x+X(t)+\sum_{i=1}^{N(t)}aU(T_i-), ~~~t \geq 0, \end{equation}
where $N(t)$ is a Poisson process with intensity $\lambda >0$.
We apply the fluctuation theory of one-sided L\'{e}vy processes to solve
the two-sided exit problems (Subsection~\ref{Bsec5.1}), and to obtain the expected discounted barrier dividends (Subsection~\ref{Bsec5.2}).
The key functions for this fluctuation theory are the scale functions; see \cite{kyprianou2006}.
To introduce these functions let us first define the Laplace exponent of $X(t)$:
\begin{equation*}\label{laplaceexponent}
\psi(\theta) := \frac{1}{t}\log \mathbb{E}_x[e^{\theta X_t}] = \eta \theta + \frac{\sigma^2}{2} \theta^2 .
\end{equation*}
This function is strictly convex, differentiable, equals zero at zero and tends to infinity at infinity.
Hence its right inverse $\Phi(q)$ exists for $q\geq 0$.
The first scale function $W^{(q)}(x)$ is the unique right-continuous function disappearing on the negative half-line whose Laplace transform is
\begin{equation}\label{scaleFunction}
\int_0^{\infty} e^{-\theta x} W^{(q)}(x) {\rm d}x = \frac{1}{\psi(\theta)-q} ,
\end{equation}
for $\theta>\Phi(q)$.
With the first scale function we can associate a second scale function via $Z^{(q)}(x):=1-q\int_{0}^{x}W^{(q)}(y)dy$.
In the case of linear Brownian motion defined in \eqref{linearBrownian} the (first) scale function for a Brownian motion with drift $\eta$ and variance $\sigma^{2}$ equals
(cf.\  \cite{kkr2013})
\[ W^{(q)}(x)
  =\frac{ 1}{\sqrt{\eta^{2}+2q\sigma^{2}}}\left[e^{(\sqrt{\eta^{2}+2q\sigma^{2}}-\eta)\frac{x}{\sigma^{2}}}-e^{-(\sqrt{\eta^{2}+2q\sigma^{2}}+\eta)\frac{x}{\sigma^{2}}}\right].
  \]
Let $\alpha<x<\beta$ and
 $$d_{\alpha}^X=\min\{t: X(t)=\alpha\}\quad\text{and}\quad u_{\beta}^X=\min\{t:X(t)=\beta\}.$$
Throughout this section we use the following three facts given in Theorem 8.1 and Theorem 8.7 of Kyprianou \cite{kyprianou2006}:
\begin{enumerate}
\item
\begin{equation}\label{eq.W} \E_{x}\left[e^{-qu_{\beta}^X}1_{u_{\beta}^X< d_{\alpha}^X}\right]=\frac{W^{(q)}(x-\alpha)}{W^{(q)}(\beta - \alpha)}. \end{equation}
\item\label{2}
\begin{equation}\label{2} \E_{x}\left[e^{-
	qd_{\alpha}^X}1_{d_{\alpha}^X<u_{\beta}^X}\right]=Z^{(q)}(x-\alpha)-\frac{W^{(q)}(x-\alpha)}{W^{(q)}(\beta - \alpha)}Z^{(q)}(\beta-\alpha).
\end{equation}
\item Let $\mathcal{E}_{q}$ be an exponentially distributed random variable with parameter $q$ independent of the process $X$. Then for $\alpha <x< \beta$:
\begin{equation}\label{eq.uq}
\frac{\prob_{x}(X(\mathcal{E}_{q})\in (y,y+{\rm d}y), \mathcal{E}_{q}<u_{\beta}^X\wedge d_{\alpha}^X)}{q{\rm d}y}=u^{(q)}_{\alpha,\beta}(x,y)=\frac{W^{(q)}(x-\alpha)}{W^{(q)}(\beta - \alpha)}W^{(q)}(\beta-y)-W^{(q)}(x-y).
\end{equation}
\end{enumerate}
\subsection{Downward exit problem and ruin time}
\label{Bsec5.1}
In this subsection we obtain
\begin{equation}\label{rhoonceagain}
\rho_{N}(x)=\E_{x}[e^{-qd_{N}}1_{d_{N}<u_{0}}].\end{equation}
This is done in three steps.
In Step 1 we determine the LST of the time, starting from some $x \in (L_{n},L_{n-1})$, to reach a level in $\mathcal{N}$ by down-crossing $L_{n-k}$, $k = 0, 1, \dots,n-1$.
In Step 2 we determine the LST of the time, starting from some $x \in (L_{n},L_{n-1})$, to reach a level in $\mathcal{N}$ by up-crossing $L_{n-k}$, $k =  1,2, \dots,n$.
In Step 3 we express $\rho_N(x)$  in $\rho_1,\dots,\rho_N$, with $\rho_n$ the LST of the time to down-cross $L_N$, starting from $L_n$, and before up-crossing $L_0$.
We construct a system of linear equations in those $\rho_n$, with the LST's of Steps 1 and 2 featuring as coefficients in those equations.
\\

\noindent
{\em Step 1: The time until the first down-crossing of $L_{n-k}$}
\\
Let $L_{n}<x<L_{n-1}$, and let $d_n^X$ and $u_{n-1}^X$ denote the times the $X$ process first down-crosses $L_n$, respectively up-crosses $L_{n-1}$, when starting from $x$. By \eqref{2} we have
\begin{eqnarray}\label{eq.xi}
&&\xi_{n}(x-L_{n}):=\E_{x}[e^{-qd_{n}^X}1_{d^X_{n}<u^X_{n-1}\wedge \mathcal{E}_{\lambda}}]\nonumber\\
&&=Z^{(q+\lambda)}(x- L_{n})-\frac{\W (x- L_{n})}{\W(L_{n-1}- L_{n})}Z^{(q+\lambda)}( L_{n-1}- L_{n}).
\end{eqnarray}
Denoting by $\tau_{L_{n-k}}^{-}$ the first time that $U$ hits a level in $\mathcal{N}$ and this is done by down-crossing  $L_{n-k}$, we derive
\begin{equation}
\tau_{ L_{n-k}}^{-}=\mathcal{E}_{1,\lambda}1_{\mathcal{E}_{1,\lambda}<u^X_{n-1}\wedge d^X_{n}}+\mathcal{E}_{2,\lambda}1_{\mathcal{E}_{2,\lambda}<u^X_{n-2}\wedge d^X_{n-1}}+...+\mathcal{E}_{k,\lambda}1_{\mathcal{E}_{k,\lambda}<u^X_{n-k}\wedge d^X_{n+1-k}}+d^X_{n-k}1_{d^X_{n-k}<\mathcal{E}_{k+1
		,\lambda}\wedge u^X_{n-k-1}},
\end{equation}
where $\mathcal{E}_{k,\lambda}, k=1,...,N$ are i.i.d.\ distributed as $\mathcal{E}_{\lambda}$.
Let
\begin{eqnarray*}
&&r_{n,n-k}(x):=\E_{x}\left[e^{-q\tau_{ L_{n-k}}^{-}}\right]\\
&&=\E_{x}\left[e^{-q(\sum_{i=1}^{k}\mathcal{E}_{i,\lambda}+d^X_{n-k})}1_{\mathcal{E}_{1,\lambda}<d^X_{n}\wedge u^X_{n-1}}1_{\mathcal{E}_{2,\lambda}<d^X_{n-1}\wedge u^X_{n-2}}...1_{\mathcal{E}_{k,\lambda}<d^X_{n-k+1}\wedge u^X_{n-k}}1_{d^X_{n-k}<\mathcal{E}_{k+1,\lambda}\wedge u^X_{n-k-1}}\right].
\end{eqnarray*}
Observe that $r_{n,n-k}(x)$ is the partial LST of the time to reach $L_{n-k}$ from above before reaching any other level in $\mathcal{N}$.
Clearly,
\begin{equation}
r_{n,n}(x)=\xi_{n}(x-L_{n}).
\label{rnnx}
\end{equation}
Applying (\ref{eq.uq}) and (\ref{rnnx}) yields:
\begin{eqnarray}
&&r_{n,n-1}(x)=\lambda\int_{L_{n}}^{L_{n-1}}u^{(q+\lambda)}_{L_{n},L_{n-1}}(x,y)r_{n-1,n-1}((a+1)y) {\rm d}y\nonumber\\
&&=\lambda \frac{W^{(q+\lambda)}(x-L_{n})}{\W (L_{n-1}-L_{n})}\int_{L_{n}}^{L_{n-1}}\W(L_{n-1}-y)\xi_{n-1}((a+1)y-L_{n-1}){\rm d}y \nonumber
\\
&&-\lambda\int_{L_{n}}^{x}\W(x-y)\xi_{n-1}((a+1)y-L_{n-1}){\rm d}y.
\label{rnn-1x}
\end{eqnarray}
Note that
\begin{eqnarray*}
	&&\int_{L_{n}}^{x}\W(x-y)\xi_{n-1}((a+1)y-L_{n-1}){\rm d}y\\
	&&=\int_{0}^{x-L_{n}}\W(z)\xi_{n-1}((a+1)(x
	-z-L_{n})){\rm d}z\\
	&&=\W\circledast\xi_{n-1,a+1}(x-L_{n}),
	\end{eqnarray*}
where $\circledast$ again denotes convolution
and $\xi_{n,(a+1)^k}(x):=\xi_{n}((a+1)^kx)$, $k\in \mathbb{N}$.
Thus:
\begin{eqnarray}\label{rnn-1}
	&&r_{n,n-1}(x)=\lambda \frac{W^{(q+\lambda)}(x-L_{n})}{\W (L_{n-1}-L_{n})}
	\W\circledast\xi_{n-1,a+1}(L_{n-1}-L_{n})-\lambda \W\circledast\xi_{n-1,a+1}(x-L_{n}).\nonumber\\\label{rnn-1}
\end{eqnarray}
Denote
\begin{eqnarray}
	A_{0,n,n-1}:=\lambda\frac{\W\circledast\xi_{n-1,a+1}(L_{n-1}-L_{n})}{\W (L_{n-1}-L_{n})}\quad\text{and}\quad
	A_{1,n,n-1}:=\lambda.
\end{eqnarray}
Then
\begin{equation*}
r_{n,n-1}(x)=A_{0,n,n-1}W^{(q+\lambda)}(x-L_{n})-A_{1,n,n-1}\W\circledast\xi_{n-1,a+1}(x-L_{n}).
\end{equation*}
We next obtain $r_{n,n-2}(x)$. Applying (\ref{eq.uq}) and (\ref{rnn-1}) we have
\begin{eqnarray*}
	&&r_{n,n-2}(x)=\lambda\int_{L_{n}}^{L_{n-1}}u^{(q+\lambda)}_{L_{n},L_{n-1}}(x,y)r_{n-1,n-2}((a+1)y){\rm d}y\nonumber\\
	&&=\lambda \frac{W^{(q+\lambda)}(x-L_{n})}{\W (L_{n-1}-L_{n})}\int_{L_{n}}^{L_{n-1}}\W(L_{n-1}-y)\cdot\\
	&&\left(A_{0,n-1,n-2}W^{(q+\lambda)}((a+1)y-L_{n-1})-A_{1,n-1,n-2}\W\circledast\xi_{n-2,a+1}((a+1)y-L_{n-1})\right){\rm d}y\nonumber\\
	&&-\lambda\int_{L_{n}}^{x}\W(x-y)\cdot\\
	&&\left(A_{0,n-1,n-2}W^{(q+\lambda)}((a+1)y-L_{n-1})-A_{1,n-1,n-2}\W\circledast\xi_{n-2,a+1}((a+1)y-L_{n-1})\right){\rm d}y.\\
\end{eqnarray*}
Similarly as before, observe that
\begin{eqnarray}
	&&\int_{L_{n}}^{x}\W(x-y)W^{(q+\lambda)}((a+1)y-L_{n-1}){\rm d}y=\W\circledast\W_{a+1}(x-L_{n}) ,
\label{convo}
\end{eqnarray}
where $\W_{(a+1)^k}(x):=\W((a+1)^kx)$, $k\in \mathbb{N}$, and
\begin{eqnarray*}
	&&\int_{L_{n}}^{x}\W(x-y)\W\circledast\xi_{n-2,a+1}((a+1)y-L_{n-1}){\rm d}y\\
	&&=(a+1)\W\circledast\W_{a+1}\circledast\xi_{n-2,(a+1)^{2}}(x-L_{n}).
\end{eqnarray*}
Denote
\begin{eqnarray}
	&&A_{0,n,n-2}:=\frac{\lambda }{\W(L_{n-1}-L_{n})}\left(A_{0,n-1,n-2}\W\circledast\W_{a+1}(L_{n-1}-L_{n})\right.
\nonumber
\\
	&&\left.\quad-A_{1,n-1,n-2}(a+1)\W\circledast\W_{a+1}\circledast\xi_{n-2,(a+1)^{2}}(L_{n-1}-L_{n})\right)
\nonumber
\\
	&&=\frac{1 }{\W(L_{n-1}-L_{n})}\cdot
\nonumber
\\
	&&\quad\left(A_{1,n,n-2}\W\circledast\W_{a+1}(L_{n-1}-L_{n})-
	A_{2,n,n-2}	\W\circledast\W_{a+1}\circledast\xi_{n-2,(a+1)^{2}}(L_{n-1}-L_{n})\right),
\nonumber
\\
	&&A_{1,n,n-2}:=\lambda A_{0,n-1,n-2},\qquad A_{2,n,n-2}:=\lambda (a+1)A_{1,n-1,n-2}.
\end{eqnarray}
Then
\begin{eqnarray*}
	&&r_{n,n-2}(x)=A_{0,n,n-2}\W(x-L_{n})-A_{1,n,n-2}\W\circledast\W_{a+1}(x-L_{n})\\
	&&\qquad+A_{2,n,n-2}\W \circledast\W_{a+1}\circledast \xi_{n-2,(a+1)^{2}}(x-L_{n}).	
\end{eqnarray*}
The general case for $k =2,\dots,n-1$ is given in the following proposition.
\begin{proposition}\label{prop6.1}
	For $L_{n}<x<L_{n-1}$ and $k =2,\dots,n-1$,
	\begin{eqnarray}\label{eq.rnk}
		&&r_{n,n-k}(x)
	=\sum_{j=0}^{k-1}(-1)^{j}A_{j,n,n-k}
	\circledast_{i=0}^{j}\W_{(a+1)^{i}}(x-L_{n})\nonumber\\
	&&\qquad+(-1)^{k}A_{k,n,n-k}\circledast_{i=0}^{k-1}\W_{(a+1)^{i}}
	\circledast\xi_{n-k,(a+1)^{k}}(x-L_{n}),
\end{eqnarray}
where $A_{j,n,n-k}$, $j=0,...,k$ are coefficients which are obtained recursively.
\end{proposition}
\begin{proof}
	The proof is by induction on $k$. Clearly, (\ref{eq.rnk}) holds for $k=2$.
	 Assume it holds for $k-1\geq2$. 
	 By the induction hypothesis we have
	\begin{eqnarray}
		&&r_{n,n-(k-1)}(x)
		=\sum_{j=0}^{k-2}(-1)^{j}A_{j,n,n-(k-1)}
		\circledast_{i=0}^{j}\W_{(a+1)^{i}}(x-L_{n})
\label{eqeq1}
\\
		&&\qquad+(-1)^{k-1}A_{k-1,n,n-(k-1)}\circledast_{i=0}^{k-2}\W_{(a+1)^{i}}
		\circledast\xi_{n-(k-1),(a+1)^{k-1}}(x-L_{n}).
\nonumber
	\end{eqnarray}
Using (\ref{eq.uq}) and (\ref{eqeq1}), we have
\begin{eqnarray*}
	&&r_{n,n-k}(x)=\lambda\int_{L_{n}}^{L_{n-1}}u^{(q+\lambda)}_{L_{n},L_{n-1}}(x,y)	r_{n-1,n-k}((a+1)y){\rm d}y\\
	&&=\lambda \frac{W^{(q+\lambda)}(x-L_{n})}{\W (L_{n-1}-L_{n})}
	\int_{L_{n}}^{L_{n-1}}\W(L_{n-1}-y)\cdot\\
	&&\Bigg(\sum_{j=0}^{k-2}(-1)^{j}A_{j,n-1,n-k}\circledast_{i=0}^{j}\W_{(a+1)^{i}}((a+1)y-L_{n-1})\\
	&&\quad+(-1)^{k-1}A_{k-1,n-1,n-k}\circledast_{i=0}^{k-2}\W_{(a+1)^{i}}\circledast\xi_{n-k,(a+1)^{k-1}}((a+1)y-L_{n-1})\Bigg){\rm d}y\\
	&&	-\lambda\int_{L_{n}}^{x}\W(x-y)\cdot\\
	&&\Bigg(\sum_{j=0}^{k-2}(-1)^{j}A_{j,n-1,n-k}\circledast_{i=0}^{j}\W_{(a+1)^{i}}((a+1)y-L_{n-1})\\
	&&+(-1)^{k-1}A_{k-1,n-1,n-k}\circledast_{i=0}^{k-2}\W_{(a+1)^{i}}\circledast\xi_{n-k,(a+1)^{k-1}}((a+1)y-L_{n-1})\Bigg){\rm d}y.\\
\end{eqnarray*}
Note that (\ref{convo}) holds, and
and for $j\geq 1$,
\begin{eqnarray*}
	&&	\int_{L_{n}}^{x}\W(x-y)\circledast_{i=0}^{j}\W_{(a+1)^{i}}((a+1)y-L_{n-1}) {\rm d}y \\
	&&\qquad =(a+1)\circledast_{i=0}^{j+1}\W_{(a+1)^{i}}(x-L_{n}).
\end{eqnarray*}

If we choose
\begin{eqnarray}\label{eq.recursionA}
	&&A_{1,n,n-k}:=\lambda A_{0,n-1,n-k},\nonumber\\
	&&A_{j+1,n,n-k}:=\lambda (a+1)A_{j,n-1,n-k},\qquad 1\leq j\leq k-1,\nonumber\\
	&&A_{0,n,n-k}:=\frac{1}{\W(L_{n-1}-L_{n}))}\left(\sum_{j=1}^{k-1} (-1)^{j-1} A_{j
		,n,n-k}\circledast_{i=0}^{j}\W_{(a+1)^{i}}(L_{n-1}-L_{n})\right.\nonumber\\
	&&\qquad+\left. (-1)^{k-1} A_{k,n,n-k} \circledast_{i=0}^{k-1}\W_{(a+1)^{i}}\circledast \xi_{n-k,(a+1)^{k}}(L_{n-1}-L_{n})\right),
\end{eqnarray}
then (\ref{eq.rnk}) holds true which completes the proof.
\end{proof}
\noindent
{\em Step 2: The time until the first up-crossing of $L_{n-k}$}
\\
Let $L_{n}<x<L_{n-1}$ and $\tau_{L_{n-k}}^{+}$ be the first time that $U$ reaches a level in $\mathcal{N}$ and it is done  by up-crossing $L_{n-k}$ by the Brownian motion. Note that
\begin{equation}
\tau_{L_{n-k}}^{+}=\mathcal{E}_{1,\lambda}1_{\mathcal{E}_{1,\lambda}<u^X_{n-1}\wedge d^X_{n}}+\mathcal{E}_{2,\lambda}1_{\mathcal{E}_{2,\lambda}<u^X_{n-2}\wedge d^X_{n-1}}+...+\mathcal{E}_{k-1,\lambda}1_{\mathcal{E}_{k-1,\lambda}<u^X_{n-k+1}\wedge d^X_{n+2-k}}+u^X_{n-k}1_{u^X_{n-k}<\mathcal{E}_{k
		,\lambda}\wedge d^X_{n-k+1}}.
\end{equation}
For $k=1,...,n$	we define
\begin{eqnarray*}
\lefteqn{\omega_{n,n-k}(x):=\E_{x}[e^{-q\tau_{L_{n-k}}^{+}}]}\\&&=
\E_{x}\left[e^{-q(\sum_{j=1}^{k-1}\mathcal{E}_{j,\lambda}+u^X_{n-k})}1_{\mathcal{E}_{1,\lambda}<u^X_{n-1}\wedge d^X_{n}}...1_{\mathcal{E}_{k-1,\lambda}<u^X_{n-k+1}\wedge d^X_{n+2-k}}1_{u^X_{n-k}<d^X_{n-k+1}\wedge\mathcal{E}_{k,\lambda} }\right].
\end{eqnarray*}
Applying (\ref{eq.W}) we have
\begin{equation}
\Omega_{n-1}(x-L_{n}):=
\omega_{n,n-1}(x)=\E_{x}(e^{-q\tau_{L_{n-1}}^{+}}1_{u^X_{n-1}<\mathcal{E}_{1,\lambda}\wedge d^X_{n}})=\frac{\W(x-L_{n})}{\W(L_{n-1}-L_{n})}.
\label{Omegan-1}
\end{equation}
Further, using (\ref{eq.uq}) and (\ref{Omegan-1}), observe that
\begin{eqnarray*}
	&&	\omega_{n,n-2}(x)=\lambda\int_{L_{n}}^{L_{n-1}}u^{(q+\lambda)}_{L_{n},L_{n-1}}(x,y)\omega_{n-1,n-2}((a+1)y){\rm d}y\\
	&&=\lambda \frac{W^{(q+\lambda)}(x-L_{n})}{\W (L_{n-1}-L_{n})}\int_{L_{n}}^{L_{n-1}}\W(L_{n-1}-y)\Omega_{n-2}((a+1)y-L_{n-1}){\rm d}y\nonumber\\
	&&\qquad-\lambda\int_{L_{n}}^{x}\W(x-y)\Omega_{n-2}((a+1)y-L_{n-1}){\rm d}y\nonumber\\
	&&=\lambda \frac{W^{(q+\lambda)}(x-L_{n})}{\W (L_{n-1}-L_{n})}\W\circledast\Omega_{n-2,a+1}(L_{n-1}-L_{n})\\
	&&\qquad-\lambda \W\circledast\Omega_{n-2,a+1}(x-L_{n}),
\end{eqnarray*}
where $\Omega_{n,(a+1)^k}(x)=\Omega_{n}((a+1)^kx)$, $k\in \mathbb{N}$.
Let
\begin{eqnarray}
&&B_{1,n,n-2}:=\lambda\quad\text{and}\quad
B_{0,n,n-2}:=\lambda\frac{\W\circledast\Omega_{n-2,a+1}(L_{n-1}-L_{n})}{\W (L_{n-1}-L_{n})}.
\end{eqnarray}
Then
\begin{eqnarray}
\omega_{n,n-2}(x)=B_{0,n,n-2}\W(x-L_{n})-B_{1,n,n-2}\W\circledast\Omega_{n-2,a+1}(x-L_{n}).
\end{eqnarray}
The next proposition gives a general expression for $\omega_{n,n-k}(x)$.
\begin{proposition}\label{prop.rho}
For $k =2,\dots,n$ we have
\begin{eqnarray}\label{eq.wnk}
&&\omega_{n,n-k}(x)
=\sum_{j=0}^{k-2}(-1)^{j}B_{j,n,n-k}
\circledast_{i=0}^{j}\W_{(a+1)^{i}}(x-L_{n})\nonumber\\
&&\qquad+(-1)^{k-1}B_{k-1,n,n-k}\circledast_{i=0}^{k-2}\W_{(a+1)^{i}}
\circledast \Omega_{ n-k,(a+1)^{k-1}}(x-L_{n}),
\end{eqnarray}
where $B_{j,n,n-k}$, $j=0,...,k-1$ are coefficients which are obtained recursively.
\end{proposition}

\begin{proof}
The proof is similar to the proof of Proposition \ref{prop6.1}.
The proposition clearly holds for $k=2$. Assume it holds for $k>2$ and $k-1$.
Applying (\ref{eq.uq}) we obtain that
\begin{eqnarray*}
	&&\omega_{n,n-k}(x)=\lambda\int_{L_{n}}^{L_{n-1}}u^{(q+\lambda)}_{L_{n},L_{n-1}}(x,y) \omega_{n-1,n-k}((a+1)y){\rm d}y\\
	&&=\lambda \frac{W^{(q+\lambda)}(x-L_{n})}{\W (L_{n-1}-L_{n})}\int_{L_{n}}^{L_{n-1}}\W(L_{n-1}-y)\cdot\\
	&&\quad\Bigg(	\sum_{j=0}^{k-3}(-1)^{j}B_{j,n-1,n-k}
	\circledast_{i=0}^{j}\W_{(a+1)^{i}}((a+1)y-L_{n-1})\\
	&&\qquad+(-1)^{k-2}B_{k-2,n-1,n-k}\circledast_{i=0}^{k-3}\W_{(a+1)^{i}}
	\circledast\Omega_{ n-k,(a+1)^{k-2}}((a+1)y-L_{n-1})\Bigg){\rm d}y\\
	&&\quad-\lambda\int_{L_{n}}^{x}\W(x-y)\cdot\\
	&&	\Bigg(\sum_{j=0}^{k-3}(-1)^{j}B_{j,n-1,n-k}
	\circledast_{i=0}^{j}\W_{(a+1)^{i}}((a+1)y-L_{n-1})\\
	&&\qquad+(-1)^{k-2}B_{k-2,n-1,n-k}\circledast_{i=0}^{k-3}\W_{(a+1)^{i}}
	\circledast\Omega_{n-k,(a+1)^{k-2}}((a+1)y-L_{n-1})\Bigg){\rm d}y.
\end{eqnarray*}
Taking
\begin{eqnarray}
&&B_{1,n,n-k}:=\lambda B_{0,n-1,n-k} , \quad
B_{j+1,n,n-k}:=(a+1)\lambda B_{j,n-1,n-k},\quad j=1,...,k-2,
\nonumber
\\
&&B_{0,n,n-k}:=\frac{1}{\W(L_{n-1}-L_{n})}\cdot
\nonumber
\\
&&\quad\Bigg(\sum_{j=1}^{k-1}(-1)^{j}B_{j,n ,n-k}\circledast_{i=0}^{j}\W_{(a+1)^{i}}(L_{n-1}-L_{n})
\nonumber
\\
&&+(-1)^{k-1} B_{k-1,n,n-k} \circledast_{i=0}^{k-2}\W_{(a+1)^{i}}\circledast\Omega_{n-k,(a+1)^{k-1}}(L_{n-1}-L_{n})\Bigg) ,
\end{eqnarray}
completes the proof of this proposition.
\end{proof}
\noindent
{\em Step 3: Determination of the exit/ruin time transform $\rho_N(x)$}
\\
To find $\rho_{N}(x)$ we start from the key observation that for $L_{n}<x<L_{n-1}$ we have
\begin{equation}\label{keyidentity}
\rho_{N}(x)=r_{n,n}(x)\rho_{n}+\sum_{j=1}^{n-1}(r_{n,n-j}(x)+\omega_{n,n-j}(x))\rho_{n-j},
\end{equation}
where
  \[\rho_{n}:=E_{L_{n}}\left[e^{-qd_N}1_{d_N<u_0}\right].\]
In the next step we construct a system of linear equations to find $\rho_{n},\,n=1,2,...,N$.
Clearly, $\rho_{0}=0$ and $\rho_{N}=1$.
Moreover,
\begin{eqnarray*}
	&&\rho_{1}=
	\left(Z^{(q+\lambda)}(L_{1}-L_{2})-\frac{\W (L_{1}-L_{2})}{\W(L_{0}-L_{2})}Z^{(q+\lambda)}(L_{0}-L_{2})\right)\rho_{2}\\
	&&\qquad+ \left(\lambda \int_{L_{2}}^{L_{1}}u^{(q+\lambda)}_{L_{2},L_{0}}(L_{1},y))r_{1,1}((a+1)y){\rm d}y \right) \rho_{1}. \\
\end{eqnarray*}
The  term in the first parentheses is the Laplace transform of the time to down-cross $L_{2}$ before $\mathcal{E}_{\lambda}$ and before $L_{0}$ is reached; cf.\ (\ref{2}).
	The second term is the Laplace transform of    $\mathcal{E}_{\lambda}$ where the exponential time expires when $U\in (y,y+dy)$ is between $L_{2}$ and $L_{1}$ before reaching $L_{2}$ or $L_{0}$ and then the time to reach $L_{1}$ from above. Similarly, note that
\begin{eqnarray}
	&&\rho_{2}=
	\left(Z^{(q+\lambda)}(L_{2}-L_{3})-\frac{\W (L_{2}-L_{3})}{\W(L_{1}-L_{3})}Z^{(q+\lambda)}(L_{1}-L_{3})\right)\rho_{3}\label{rho2d}\\
	&&\qquad+\frac{\W(L_{2}-L_{3})}{\W(L_{1}-L_{3})}\rho_{1}\label{rho2u}\\
&&\qquad+\lambda\int_{L_{3}}^{L_{2}}u^{(q+\lambda)}_{L_{3},L_{1}}(L_{2},y)\left(r_{2,2}((a+1)y)	\rho_{2}+(r_{2,1}((a+1)y)+\omega_{2,1}((a+1)y))\rho_{1}\right){\rm d}y\label{rho2s2}\\
&&\qquad+ \left( \lambda\int_{L_{2}}^{L_{1}}u^{(q+\lambda)}_{L_{3},L_{1}}(L_{2},y)r_{1,1}((a+1)y){\rm d}y \right) \rho_1.\label{rho2s1}
\end{eqnarray}
The term in the parentheses in (\ref{rho2d}) is the expected discounted time to reach $L_{3}$ before a jump and before up-crossing $L_{1}$. The factor in (\ref{rho2u}) is the expected discounted time  to reach $L_{1}$
before a jump and before down-crossing $L_{3}$ (cf.\ (\ref{eq.W})). (\ref{rho2s2}) and (\ref{rho2s1}) describe the expected discounted time until a jump when a jump occurs before   reaching $L_{1}$ or $L_{3}$ and then the expected discounted time until the process reaches one of the levels $L_{j}$ for  $j\leq 2$.
  (\ref{rho2s2}) describes the case where just before a jump $U$ is between $L_{3}$ and $L_{2}$ and( \ref{rho2s1}) describes the case where just before a jump $U$ is between $L_{2}$ and $L_{1}$.
 By rearranging (\ref{rho2d})-(\ref{rho2s1}) we get
\begin{eqnarray*}
	&&\rho_{2}=
\left(Z^{(q+\lambda)}(L_{2}-L_{3})-\frac{\W (L_{2}-L_{3})}{\W(L_{1}-L_{3})}Z^{(q+\lambda)}(L_{1}-L_{3})\right)\rho_{3}\\\
&&+\left( \lambda\int_{L_{3}}^{L_{2}}u^{(q+\lambda)}_{L_{3},L_{1}}(L_{2},y)r_{2,2}((a+1)y){\rm d}y \right) \rho_{2}\\
&&+\left(\frac{\W(L_{2}-L_{3})}{\W(L_{1}-L_{3})}+\lambda \int_{L_{3}}^{L_{2}}u^{(q+\lambda)}_{L_{3},L_{1}}(L_{2},y)(r_{2,1}((a+1)y)+
\omega_{2,1}((a+1)y){\rm d}y\right.\\
		&&+\left. \lambda \int_{L_{2}}^{L_{1}}u^{(q+\lambda)}_{L_{3},L_{1}}(L_{2},y)r_{1,1}((a+1)y){\rm d}y\right)\rho_{1}.
\end{eqnarray*}
Using similar arguments, we can show that generally, for $1<n\leq N-1$,
\begin{eqnarray*}
	&&	\rho_{n}=\left(Z^{(q+\lambda)}(L_{n}-L_{n+1})-\frac{\W (L_{n}-L_{n+1})}{\W(L_{n-1}-L_{n+1})}Z^{(q+\lambda)}(L_{n-1}-L_{n+1})\right)\rho_{n+1}\\
	&&\qquad+\frac{\W(L_{n}-L_{n+1})}{\W(L_{n-1}-L_{n+1})}\rho_{n-1}\\
	&&\qquad+\lambda\int_{L_{n+1}}^{L_{n}}u^{(q+\lambda)}_{L_{n+1},L_{n-1}}(L_{n},y)\left(r_{n,n}((a+1)y)\rho_{n}+\sum_{k=1}^{n-1}
	\left( r_{n,n-k}((a+1)y)+\omega_{n,n-k}((a+1)y) \right) \rho_{n-k}\right){\rm d}y\\
&&\qquad+\lambda\int_{L_{n}}^{L_{n-1}}u^{(q+\lambda)}_{L_{n+1},L_{n-1}}(L_{n},y)\Bigg(r_{n-1,n-1}((a+1)y)\rho_{n-1}\\
&&\qquad\qquad+\sum_{k=1}^{n-2}
	(r_{n-1,n-1-k}((a+1)y)+\omega_{n-1,n-1-k}((a+1)y)\rho_{n-1-k}\Bigg){\rm d}y ,
	\end{eqnarray*}
which is equivalent to
\begin{eqnarray}
&&\rho_{n}=\left(Z^{(q+\lambda)}(L_{n}-L_{n+1})-\frac{\W (L_{n}-L_{n+1})}{\W(L_{n-1}-L_{n+1})}Z^{(q+\lambda)}(L_{n-1}-L_{n+1})\right)\rho_{n+1}\nonumber\\
&&\quad+ \left( \lambda\int_{L_{n+1}}^{L_{n}}u_{L_{n+1},L_{n-1}}^{(q+\lambda)}(L_{n},y)r_{n,n}((a+1)y){\rm d}y \right) \rho_{n}\nonumber\\
&&\quad+\left(\lambda \int_{L_{n+1}}^{L_{n}}u^{(q+\lambda)}_{L_{n+1},L_{n-1}}(L_{n},y)(r_{n,n-1}((a+1)y)+\omega_{n,n-1}((a+1)y)){\rm d}y\right.\nonumber\\
&&\quad+\left.\frac{\W(L_{n}-L_{n+1})}{\W(L_{n-1}-L_{n+1})}+\lambda \int_{L_{n}}^{L_{n-1}}u^{(q+\lambda)}_{L_{n+1},L_{n-1}}(L_{n},y)r_{n-1,n-1}((a+1)y){\rm d}y\right)\rho_{n-1}\nonumber\\
&&\quad+\lambda\sum_{k=2}^{n-1}\left(\int_{L_{n+1}}^{L_{n}}u^{(q+\lambda)}_{L_{n+1},L_{n-1}}(L_{n},y)(r_{n,n-k}((a+1)y)+\omega_{n,n-k}((a+1)y)){\rm d}y\right.\nonumber\\
&&\quad+\left.\int_{L_{n}}^{L_{n-1}}u^{(q+\lambda)}_{L_{n+1},L_{n-1}}(L_{n},y)(r_{n-1,n-1-(k-1)}((a+1)y)+\omega_{n-1,n-1-(k-1)}((a+1)y)){\rm d}y\right)\rho_{n-k}.\nonumber\\
\label{systemforrho}
\end{eqnarray}
Thus we have proved the following main result.
\begin{theorem}
The two-sided downward exit time transform $\rho_{N}(x)$ defined in
\eqref{rhoonceagain} is given in \eqref{keyidentity} with $r_{n,n-k}$
identified in (\ref{rnnx}), (\ref{rnn-1x}) and Proposition \ref{prop6.1}, $\omega_{n,n-k}$ identified in (\ref{Omegan-1}) and Proposition \ref{prop.rho}, and
with $\rho_k$ given via the system of equations \eqref{systemforrho}.
\end{theorem}

\subsection{Expected discounted dividends until ruin}
\label{Bsec5.2}
In this section we obtain $v_{N}(x)$ -- the expected discounted dividends obtained until the process reaches $L_{N}$ starting at $x$.
%
Let $L_{n}<x<L_{n-1}$ and
\begin{equation}
\mathcal{T}_{n,0}(x):=\E_{x}[e^{-q\mathcal{S}_{n}}1_{\mathcal{E}_{1,\lambda}<u^X_{n-1}\wedge d^X_{n}}1_{\mathcal{E}_{2,\lambda}<u^X_{n-2}\wedge d^X_{n-1}}...1_{\mathcal{E}_{n,\lambda}<u^X_{0}\wedge d^X_{1}} ],
\end{equation}
where
\[\mathcal{S}_{n}:=\sum_{i=1}^{n}\mathcal{E}_{i,\lambda}.\]
Thus $\mathcal{T}_{n,0}(x)$ is the expected discounted time until up-crossing $L_0$ by a jump when it occurs before reaching any level in $\mathcal{N}$.
Also for $L_{n}<x<L_{n-1}$ let
\begin{equation}
v_{n}^{J}(x):=\E_{x}[e^{-qS_n}1_{\mathcal{E}_{1,\lambda}<u^X_{n-1}\wedge d^X_{n}}1_{\mathcal{E}_{2,\lambda}< u_{n-2}^X \wedge d^X_{n-1}}...1_{\mathcal{E}_{n,\lambda} < u_0^X \wedge d^X_{1}}((a+1)U(\mathcal{S}_{n})- L_0) ].
\label{vnJdef}
\end{equation}
Note that $v_{n}^{J}(x)$ is the expected discounted overflow above $L_0=b$ when it occurs before reaching any level in $\mathcal{N}$.
First consider $v_1^J(x)$, so take $L_{1}<x<L_{0}$. Applying (\ref{eq.uq}) we get
\begin{eqnarray}
&&v_{1}^{J}(x)=\lambda\int_{L_{1}}^{L_{0}}u^{(q+\lambda)}_{L_{1},L_{0}}(x,y)((a+1)y- L_0){\rm d}y
\label{5.27}
\\
&&=\lambda\left(\frac{\W(x-L_{1})}{\W(L_{0}-L_{1})}\int_{L_{1}}^{L_0}\W(L_{0}-y)((a+1)y- L_0){\rm d}y-\int_{L_{1}}^{x}\W(x-y)((a+1)y- L_0){\rm d}y\right).	
\nonumber
\end{eqnarray}
Thus,
\begin{equation*}
v_{1}^{J}(x)=A_{1,f,0}\W(x-L_{1})-A_{1,f,1}\W\circledast\mathcal{Q}_{a+1}(x-L_{1}) ,
\end{equation*}
where $\mathcal{Q}_{(a+1)^k}(x)=\mathcal{Q}((a+1)^kx)=(a+1)^kx$, $k\in \mathbb{N}$, and
\begin{eqnarray}
&&A_{1,f,0}:=\lambda\frac{\W\circledast\mathcal{Q}_{a+1}(L_{0}-L_{1})}{\W(L_{0}-L_{1})}\quad\text{and}\quad
A_{1,f,1}:=\lambda.
\end{eqnarray}
Similarly, replacing $(a+1)y- L_0$ by $1$ in (\ref{5.27}), we obtain that the expected discounted time until a jump above $L_0=b$ is:
\begin{equation}
\mathcal{T}_{1,0}(x)=A_{1,\mathcal{T},0}\W(x-L_{1})-A_{1,\mathcal{T},1}\overline{W}^{(q+\lambda)}(x-L_{1}),
\end{equation}
where
\begin{equation}
A_{1,\mathcal{T},0}=\lambda\frac{\overline{W}^{(q+\lambda)}(L_{0}-L_{1})}{\W(L_{0}-L_{1})} \quad \text{and} \quad
A_{1,\mathcal{T},1}=\lambda, \label{eq.at2}
\end{equation}
and $\overline{W}^{(q)}(x)=\int_{0}^{x}W^{(q)}(y)dy$.

Next consider $v_2^J(x)$, so take $L_{2}<x<L_{1}$. Then
\begin{eqnarray*}
&&v^{J}_{2}(x)=\lambda\int_{L_{2}}^{L_{1}}u^{(q+\lambda)}_{L_{2},L_{1}}(x,y)v_{1}^{J}((a+1)y){\rm d}y\\
&&=\lambda\frac{\W(x-L_{2})}{\W(L_{1}-L_{2})}\int_{L_{2}}^{L_{1}}\W(L_{1}-y)\left(A_{1,f,0}\W((a+1)y-L_{1})\right.\\
&&\left.\qquad-A_{1,f,1}\W\circledast\mathcal{Q}_{a+1}((a+1)y-L_{1})\right){\rm d}y\\
&&\qquad-\lambda \int_{L_{2}}^{x}\W(x-y)\left(A_{1,f,0}\W((a+1)y-L_{1})\right.\\
&&\left.\qquad-A_{1,f,1}\W\circledast\mathcal{Q}_{a+1}((a+1)y-L_{1})\right){\rm d}y.\\
\end{eqnarray*}
Thus,
\begin{equation}
v^{J}_{2}(x)=A_{2,f,0}\W(x-L_{2})-A_{2,f,1}\W\circledast\W_{a+1}(x-L_{2})+A_{2,f,2}\W\circledast\W_{a+1}\circledast\mathcal{Q}_{(a+1)^{2}}(x-L_{2}),
\end{equation}
where
\begin{eqnarray}
&&A_{2,f, 1}:=\lambda A_{1,f,0}, \quad  \quad  A_{2,f,2}:=\lambda (a+1) A_{1,f,1}\nonumber\\
&&A_{2,f,0}:=\frac{\lambda(A_{1,f,0}\W\circledast\W_{a+1}(L_{1}-L_{2})-A_{1,f,1}(a+1)\W\circledast\W_{a+1}\circledast\mathcal{Q}_{(a+1)^{2}}(L_{1}-L_{2}))}{\W(L_{1}-L_{2})}
\nonumber\\
&&=\frac{A_{2,f,1}\W\circledast\W_{a+1}(L_{1}-L_{2})-A_{2,f,2}\W\circledast\W_{a+1}\circledast\mathcal{Q}_{(a+1)^{2}}(L_{1}-L_{2})}{\W(L_{1}-L_{2})}
\label{A2f1}
\end{eqnarray}
Similarly,
\begin{equation}
\mathcal{T}_{2,0}(x)=A_{2,\mathcal{T},0}\W(x-L_{2})-A_{2,\mathcal{T},1}\W\circledast\W_{a+1}(x-L_{2})+A_{2,\mathcal{T},2}\W\circledast\overline{W}_{a+1}(x-L_{2}),
\end{equation}	
where
\begin{eqnarray}
&&A_{2,\mathcal{T},1}:=\lambda A_{1,\mathcal{T},0}  \quad \text{and} \quad
A_{2,\mathcal{T},2}:=\lambda A_{1,\mathcal{T},1},\label{eq.at22}\\
&&A_{2,\mathcal{T},0}:=\frac
{ A_{2,\mathcal{T},1}\W \circledast \W_{a+1}(L_{1}-L_{2})-A_{2,\mathcal{T},2} \W \circledast\overline{W}^{(q+\lambda)}_{a+1}(L_{1}-L_{2})}{\W(L_{1}-L_{2})}. \label{eq.at20}
\end{eqnarray}
Using similar arguments like in the proof of Proposition \ref{prop.rho}
one can derive the following result.
\begin{proposition}\label{vjn}
For $L_{n}<x<L_{n-1}$ we have
	\begin{eqnarray*}
	&&v^{J}_{n}(x)=\sum_{j=0}^{n-1}(-1)^{j}A_{n,f,j}\circledast_{i=0}^{j}\W_{(a+1)^{i}}(x-L_{n})+(-1)^{n} A_{n,f,n} \circledast_{i=0}^{n-1} \W_{(a+1)^{i}}\circledast \mathcal{Q}_{(a+1)^{n}}(x-L_{n}),
	\end{eqnarray*}
where $A_{n,f,j}$ are obtained recursively as follows:
	\begin{eqnarray}
	&&A_{n,f,1}:=\lambda 	A_{n-1,f,0},\\
&&A_{n,f,j}:=\lambda(a+1)A_{n-1,f,j-1},\,\,\,2\leq j\leq n,\\
&&A_{n,f,0}:=\frac{\sum_{j=1}^{n-1}(-1)^{j-1}A_{n ,f,j}\circledast_{i=0}^{j}\W_{(a+1)^{i}}(L_{n-1}-L_{n})}{\W(L_{n-1}-L_{n})}
\nonumber
\\
&&\qquad\qquad
+(-1)^{n-1}\frac{A_{n ,f,n}\circledast_{i=0}^{n-1}\W_{(a+1)^{i}}\circledast\mathcal{Q}_{(a+1)^{n}}(L_{n-1}-L_{n})}{\W(L_{n-1}-L_{n})}.
	\end{eqnarray}
Similarly, for $L_{n}<x<L_{n-1}$,
\begin{eqnarray*}
&&\mathcal{T}_{n,0}(x)=\sum_{j=0}^{n-1}(-1)^{j}A_{n,\mathcal{T},j}\circledast_{i=0}^{j-1}\W_{(a+1)^{i}}(x-L_{n})
+(-1)^{n} A_{n,\mathcal{T},n} \circledast_{i=0}^{n-1}\W_{(a+1)^{i}}\circledast\overline{W}^{(q+\lambda)}_{(a+1)^{n}}(x-L_{n}),
\end{eqnarray*}
where
$A_{n,\mathcal{T},j}$, $n=1,2$, $j=0,1,2$ are as in \eqref{eq.at2} and \eqref{eq.at20}- \eqref{eq.at22}.
For $n \geq 3$, $A_{n,\mathcal{T},j}$ are obtained recursively as follows:
\begin{eqnarray}
	&&A_{n,\mathcal{T},1}:=\lambda 	A_{n-1,\mathcal{T},0},\\
	&&A_{n,\mathcal{T},j}:=\lambda(a+1)A_{n-1,\mathcal{T},j-1},\,\,\,2\leq j\leq n, \\
&&A_{n,\mathcal{T},0}:=\frac{\sum_{j=1}^{n-1}(-1)^{j-1}A_{n,\mathcal{T},j}\circledast_{i=0}^{j}\W_{(a+1)^{i}}(L_{n-1}-L_{n})}{\W(L_{n-1}-L_{n})}
\nonumber
\\&&
\qquad\qquad
+(-1)^{n-1}\frac{A_{n,\mathcal{T},n}\circledast_{i=0}^{n-1}\W_{(a+1)^{i}}\circledast\overline{W}^{(q+\lambda)}_{(a+1)^{n}}(L_{n-1}-L_{n})}{\W(L_{n-1}-L_{n})}.
\end{eqnarray}
\end{proposition}
\noindent
Recall that $v_N(x)$ is the expected discounted dividends starting at state $x$ until reaching $L_N$, and (cf.\ (\ref{vndef})) $v_n = v_N(L_n)$.
Observe that for $L_{n}<x<L_{n-1}$ we have
\begin{eqnarray}
&&v_N (x)=\sum_{k=0}^{n-1}r_{n,n-k}(x)v_{n-k}+\sum_{k=1}^{n-1}\omega_{n,n-k}(x)v_{n-k}\nonumber\\&&
\qquad +(\omega_{n,0}(x)+\mathcal{T}_{n,0})v_{0} + v_n^J(x).\label{secondkeyidentity}
\end{eqnarray}
The first two terms in the righthand side of (\ref{secondkeyidentity})
correspond to cases in which a level from $\mathcal{N}$ is reached before $L_0=b$ is reached or up-crossed.
The $v_0$ term covers the two cases in which level $L_0$ is reached
($\omega_{n,0}(x)$ is the expected discounted time to  reach level
		$L_{0}$ before reaching any other level in $\mathcal{N}$) and level $L_0$ is up-crossed by a jump ($\mathcal{T}_{n,0}(x)$ is the expected discounted time until up-crossing $L_{0}$ before reaching any other level in $\mathcal{N}$).
Finally $v_n^J(x)$ is the expected discounted overflow above $L_0$ by a jump, when it occurs before reaching any level in $\mathcal{N}$.

We now derive a system of equations identifying all $v_n$.
Clearly $v_{N}=0$. Let us set an equation for $v_{0}$.
Assume that $U(0)=b=L_{0}$.
Let $\overline{X}(t):=\sup_{ 0\leq s\leq  t}\{X(s)\}$, $V(t):=(\overline{X}(t)-b)_{+}$
and let $R(t)=X(t)-V(t)$, where $y_{+}=\max(y,0)$. Observe that $V(t)$ is the cumulative amount of
dividends obtained up to time $t$ only via process $X(t)$.
From Theorem 8.11 in Kyprianou \cite{kyprianou2006} we have, with $d^R_ {\alpha} = {\rm min}\{t: R(t) = \alpha\}$:
\begin{equation}\label{eq.muq}
\frac{\prob_{x}(R(\mathcal{E}_{q})\in (y,y+{\rm d}y), \mathcal{E}_{q}< d^{R}_{\alpha})} {q{\rm d}y}=\mu^{(q)}(x,y)=\frac{W^{(q)}(x-\alpha)}{W^{(q)'
	}(b - \alpha)}W^{(q)'}(b-y)-W^{(q)}(x-y),
\end{equation}
where $W^{(q)'}(x)$ is the derivative of $W^{(q)}(x)$ with respect to $x$.
Moreover, from \cite{APP} we know that the expected discounted dividends paid until $d_{L_1}^{R} \wedge \mathcal{E}_{\lambda}$ starting at $b$
equals
\begin{equation} \eta\left(b,\frac{b}{a+1}\right):=\E_{b}\left[\int_{0}^{\infty}e^{-qt}1_{t< d_{L_1}^R \wedge\mathcal{E}_{\lambda}}{\rm d}V(t)\right]=\frac{W^{(q+\lambda)}(b-\frac{b}{a+1})}{W^{(q+\lambda)'}(b-\frac{b}{a+1})}.
	\end{equation}
Additionally, from Theorem 8.10(i) in Kyprianou \cite{kyprianou2006} with $\theta=0$ we have
\begin{equation}\label{reflectedexitLT}
\E_x\left[e^{-q d^R_\alpha}1_{d^R_{\alpha}<\mathcal{E}_{\lambda}}\right]=
Z^{(q+\lambda)}(x-\alpha)-(q+\lambda)\frac{\W(b-\alpha)}{W^{(q+\lambda)'}(b-\alpha)}\W(x-\alpha).
\end{equation}
Therefore:
	 \begin{eqnarray}
	 &&v_{0}=	\eta\left(b,\frac{b}{a+1}\right)+\lambda\int_{L_{1}}^{L_{0}}\mu^{(q+\lambda)}(b,y)((a+1)y-b+v_{0}){\rm d}y\nonumber\\
	 	&&\qquad+\left(Z^{(q+\lambda)}(L_{0}-L_{1})-(q+\lambda)\frac{\W(L_{0}-L_{1})}{W^{(q+\lambda)'}(L_{0}-L_{1})}\W(L_{0}-L_{1})\right)v_{1}.\label{vzero}
	 \end{eqnarray}
The second term is the expected discounted dividends due to a jump that occurs at time $\mathcal{E}_{\lambda}$ before down-crossing $L_{1}$.
The last term
equals $\E_{L_0}\left[e^{-q d^R_{L_1}}1_{d^R_{L_1}<\mathcal{E}_{\lambda}}\right]v_1$ and hence
is the expected discounted dividends when $L_{1}$ is down-crossed before the exponential time $\mathcal{E}_{\lambda}$  has expired.
Further, we have
\begin{eqnarray}
&&v_{1}=
\left(Z^{(q+\lambda)}(L_{1}-L_{2})-\frac{\W (L_{1}-L_{2})}{\W(L_{0}-L_{2})}Z^{(q+\lambda)}(L_{0}-L_{2})\right)v_{2}\label{v12}\\
&&\qquad+\left( \lambda\int_{L_{2}}^{L_{1}}u^{(q+\lambda)}_{L_{2},L_{0}}(L_{1},y) r_{1,1}((a+1)y){\rm d}y\,\right) v_{1}\label{v11}\\
&&\qquad+\left(\frac{\W(L_{1}-L_{2})}{\W(L_{0}-L_{2})}\label{v10a}\right.\label{v10W}\\
&&+
\lambda\left.\int_{L_{2}}^{L_{1
}}u^{(q+\lambda)}_{L_{2},L_{0}}(L_{1},y)\left(\mathcal{T}_{1,0}((a+1)y) +\omega_{1,0}((a+1)y)\right){\rm d}y\right)\,v_{0}
\label{v10}\\
&&\qquad+\lambda\int_{L_{1}}^{L_{0}}u^{(q+\lambda)}_{L_{2},L_{0}}(L_{1},y)((a+1)y-b){\rm d}y\label{v10d}\\
&&\qquad+\lambda\int_{L_{2}}^{L_{1}}u^{(q+\lambda)}_{L_{2},L_{0}}(L_{1},y)v_{1}^{J}((a+1)y
){\rm d}y . \label{v10d1}
\end{eqnarray}
The term in the parentheses in (\ref{v12}) is the expected discounted time to reach $L_{2}$ before a jump and before reaching $L_{0}$.
The term that multiplies $v_{1}$ in (\ref{v11}) is the expected discounted time to reach $L_{1}$ before any other level in $\mathcal{N}$ is reached. (\ref{v10W}) is the expected discounted time to reach $L_{0}$ by the Brownian motion before down-crossing $L_{2}$ and before the exponential time has expired. The first term in (\ref{v10}) is the expected discounted time to jump above $b$ when this jump occurs before the exponential time has expired and the second term is the expected discounted time to reach $b$ by the Brownian motion.
(\ref{v10d}) is the expected discounted dividends due to a jump when the exponential time has expired while the process is in $(L_{1},L_{0})$ and (\ref{v10d1}) is the expected discounted dividends due to a jump
when the exponential time has expired while the process is in $(L_2,L_1)$. Similarly, we can observe that
\begin{eqnarray}
&&v_{2}=
\left(Z^{(q+\lambda)}(L_{2}-L_{3})-\frac{\W (L_{2}-L_{3})}{\W(L_{1}-L_{3})}Z^{(q+\lambda)}(L_{1}-L_{3})\right)v_{3}\label{v2d}\\
&&\qquad+\left( \lambda\int_{L_{3}}^{L_{2}}u^{(q+\lambda)}_{L_{3},L_{1}}(L_{2},y)r_{2,2}((a+1)y){\rm d}y\, \right) v_{2}\label{v22}\\
&&\qquad+\left(\frac{\W(L_{2}-L_{3})}{\W(L_{1}-L_{3})}+\lambda\int_{L_{3}}^{L_{2}}u^{(q+\lambda)}_{L_{3},L_{1}}(L_{2},y)(r_{2,1}((a+1)y)+\omega_{2,1}((a+1)y)){\rm d}y \right. \label{v21a}\\
&&\qquad+\left.\lambda\int_{L_{2}}^{L_{1}}u^{(q+\lambda)}_{L_{3},L_{1}}(L_{2},y)r_{1,1}((a+1)y){\rm d}y\right) v_{1}  \label{v21}\\
&&\qquad+\left(\lambda\int_{L_{3}}^{L_{2}}u^{(q+\lambda)}_{L_{3},L_{1}}(L_{2},y)(\mathcal{T}_{2,0}((a+1)y)+\omega_{2,0}((a+1)y)){\rm d}y \right. \label{v20a}\\
&&\qquad+\left.\lambda\int_{L_{2}}^{L_{1}}u^{(q+\lambda)}_{L_{3},L_{1}}(L_{2},y)(\mathcal{T}_{1,0}((a+1)y)+\omega_{1,0}(a+1)y)){\rm d}y\right)v_{0}\label{v20}\\
&&\qquad+\lambda\left(\int_{L_{3}}^{L_{2}}u^{(q+\lambda)}_{L_{3},L_{1}}(L_{2},y)v_{2}^{J}((a+1)y){\rm d}y+\int_{L_{2}}^{L_{1}}u^{(q+\lambda)}_{L_{3},L_{1}}(L_{2},y)v_{1}^{J}((a+1)y){\rm d}y\right).\label{v2f}
\end{eqnarray}
Indeed, (\ref{v2d}) is the expected discounted dividends when level $L_{3}$ is reached before any other level in $\mathcal{N}$ and before a jump. (\ref{v22}) is the expected discounted dividends when a jump occurs before reaching $L_{1}$ or $L_{3}$, and just before the jump $U$ is between $L_{2}$ and $L_{3}$. Similarly, (\ref{v21a}) and (\ref{v21}) are the expected discounted dividends when a jump occurs before reaching $L_{1}$ or $L_{3}$ and after this jump the first level that is reached is $L_{1}$. Additionally, (\ref{v20a}) and (\ref{v20}) are the expected discounted dividends when a jump occurs before reaching $L_{1}$ or $L_{3}$ and after this jump the first level that is reached is $L_{0}$.
Moreover, (\ref{v2f}) is the expected discounted dividend due to overflow  above $b$ when a jump occurs before reaching $L_{1}$ or $L_{3}$ and after this jump the first level that is reached is $L_{0}=b$ due to dividends payment after up-crossing $b$ by a jump.
Using similar arguments we can conclude that for $2\leq n\leq N-1$ we have:
\begin{eqnarray}\label{vnequations}
\lefteqn{v_{n}=\left(Z^{(q+\lambda)}(L_{n}-L_{n+1})-\frac{\W (L_{n}-L_{n+1})}{\W(L_{n-1}-L_{n+1})}Z^{(q+\lambda)}(L_{n-1}-L_{n+1})\right)v_{n+1}}\label{vn1}\\
&&+\left( \lambda\int_{L_{n+1}}^{L_{n}}u^{q+\lambda)}_{L_{n+1},L_{n-1}}(L_{n},y)r_{n,n}((a+1)y){\rm d}y\, \right) v_{n}\label{vn2}\\
&&+\left(\lambda \int_{L_{n+1}}^{L_{n}}u^{(q+\lambda)}_{L_{n+1},L_{n-1}}(L_{n},y))(r_{n,n-1}((a+1)y)+\omega_{n,n-1}((a+1)y)){\rm d}y\right.\label{vn3}\\
&&+\frac{\W(L_{n}-L_{n+1})}{\W(L_{n-1}-L_{n+1})}
+\left.\lambda\int_{L_{n}}^{L_{n-1}}u^{(q+\lambda)}_{L_{n+1},L_{n-1}}(L_{n},y)r_{n-1,n-1}((a+1)y){\rm d}y\right)\,v_{n-1}\label{vn4}\\
&&+\lambda\sum_{k=2}^{n-1}\left(\int_{L_{n+1}}^{L_{n}}u^{(q+\lambda)}_{L_{n+1},L_{n-1}}(L_{n},y)(r_{n,n-k}((a+1)y)+\omega_{n,n-k}((a+1)y)){\rm d}y\right.\label{vn5}\\
&&\left.+\int_{L_{n}}^{L_{n-1}}u^{(q+\lambda)}_{L_{n+1},L_{n-1}}(L_{n},y)(r_{n-1,n-1-(k-1)}((a+1)y)+\omega_{n-1,n-1-(k-1)}((a+1)y)){\rm d}y\right)v_{n-k}\nonumber\\\label{vn6}\\
&&+\lambda\left( \int_{L_{n+1}}^{L_{n}}u^{(q+\lambda)}_{L_{n+1},L_{n-1}}(L_{n},y)\omega_{n,0}((a+1)y){\rm d}y+\int_{L_{n}}^{L_{n-1}}u^{(q+\lambda)}_{L_{n+1},L_{n-1}}(L_{n},y)\omega_{n-1,0}((a+1)y){\rm d}y\right.\nonumber\\
\label{vn7}\\
&&+\left.\int_{L_{n+1}}^{L_{n}}u^{(q+\lambda)}_{L_{n+1},L_{n-1}}(L_{n},y)\mathcal{T}_{n,0}((a+1)y){\rm d}y+\int_{L_{n}}^{L_{n-1}}u^{(q+\lambda)}_{L_{n+1},L_{n-1}}(L_{n},y)\mathcal{T}_{n-1,0}((a+1)y){\rm d}y\right)v_{0}\nonumber\\ \label{vn8}\\
&&+\lambda\left(\int_{L_{n+1}}^{L_{n}}u^{(q+\lambda)}_{L_{n+1},L_{n-1}}(L_{n},y)v^{J}_{n}((a+1)y){\rm d}y+\int_{L_{n}}^{L_{n-1}}u^{(q+\lambda)}_{L_{n+1},L_{n-1}}(L_{n},y)v^{J}_{n-1}((a+1)y){\rm d}y\right).\nonumber\\
\label{vn9}
\end{eqnarray}
(\ref{vn1})-(\ref{vn6}) are obtained by the same arguments  as those leading to (\ref{systemforrho}). (\ref{vn7})-(\ref{vn8}) describes the expected discounted time  to reach $L_{0}$ by the Brownian motion when it is the first level reached in $\mathcal{N}$. In (\ref{vn7}) level $L_{0}$ is reached by the Brownian motion and in (\ref{vn8}) it is reached immediately after a jump above $L_{0}$. Finally,  (\ref{vn9}) describes the expected discounted dividends paid due to up-crossing of $L_{0}$ when it occurs before any other level in $\mathcal{N}$  has been reached.
Finally, notice that
\begin{equation}\label{lasteq}v_{N}=0. \end{equation}
To sum up, we have the following main result.
\begin{theorem}
The value function $v_N(x)$, defined formally in \eqref{vN}, is given by
\eqref{secondkeyidentity} with $r_{n,n-k}$,
$\omega_{n,n-k}$,
identified in Propositions \ref{prop6.1} and \ref{prop.rho}
and $v^{J}_{n}(x)$, $\mathcal{T}_{n,0}(x)$
in  Proposition \ref{vjn},
and
with $v_n$ solving the system of linear equations \eqref{vzero} - \eqref{lasteq}.
\end{theorem}

\section{Suggestions for further research}
\label{suggest}
The present study might serve as a first step towards the analysis of more general classes of insurance models related with proportional gains.
Below we suggest a few topics for further research.
\\
(i) One could consider more general jumps up from level $u$, possibly of the form $u + \zeta(u) + C_i$,
where $\zeta(u)$ is a subordinator.
\\
(ii) In Sections~\ref{sec5} and \ref{sec:Brownian} we have considered proportional growth at jump epochs, assuming that $C_i \equiv 0$.
It would be interesting to remove the latter assumption.
\\
(iii) Another interesting research topic is an exact analysis
of the value function $v(x)$ defined in \eqref{vpi}, without taking recourse to the approximation approach with levels $L_0,\dots,L_N$.
One would then have to solve the differential-delay equation (\ref{eqeq}).

\noindent
{\bf Acknowledgment}.
The authors are grateful to Eurandom (Eindhoven, The Netherlands) for organizing the {\it Multidimensional Queues, Risk and Finance Workshop},
where this project started.

\end{document}